\documentclass[a4paper,numbook,babel,final,envcountsame,11pt]{article}

\usepackage{amssymb}
\usepackage{amsthm}
\usepackage{amsmath}
\usepackage{graphicx}
\usepackage{array,hhline}

\newtheorem{lemma}{Lemma}[section]
\newtheorem{theorem}[lemma]{Theorem}
\newtheorem{proposition}[lemma]{Proposition}
\newtheorem{conjecture}[lemma]{Conjecture}
\newtheorem{corollary}[lemma]{Corollary}
\theoremstyle{definition}
\newtheorem{definition}[lemma]{Definition}
\newtheorem{remark}[lemma]{Remark}

\numberwithin{equation}{section}
\numberwithin{figure}{section}

\begin{document}

\title{\huge On the basic properties of $GC_n$ sets}         
\author{Hakop Hakopian, Navasard Vardanyan}        
\date{}          

\maketitle



\begin{abstract}
      In this paper the simplest $n$-correct sets in the plane - $GC_n$ sets are studied.
An $n$-correct node set $\mathcal X$ is called  $GC_n$ set if the fundamental polynomial of each node is a product of $n$ linear factors. We say that a node uses a line if the line is a factor of the fundamental polynomial of this node.  A line is called $k$-node line if it passes through
exactly $k$ nodes of $\mathcal X.$ At most $n+1$ nodes can be collinear in any $GC_n$ set and an $(n+1)$-node line is called a maximal line.
The Gasca-Maeztu conjecture (1982) states that for every $GC_n$ set there exists at least one maximal line. Until now the conjecture has been
proved only for the cases $n \le 5.$

Here, for a line $\ell$ we introduce and study the concept of $\ell$-lowering of the set $\mathcal X$ and define so called proper lines.  We also provide refinements of several basic
properties of $GC_n$ sets regarding the maximal lines, $n$-node lines, the used lines, as well as the subset of nodes that use a given line.
\end{abstract}




\section{Introduction\label{sec:intro}}
An \emph{$n$-correct set} $\mathcal X$ in the plane is a node set for which the Lagrange interpolation problem with bivariate polynomials of total degree not exceeding $n$ is unisolvent.
A line is called \emph{$k$-node line} if it passes through
exactly $k$ nodes of $\mathcal X.$ It is a simple fact that at most $n+1$ nodes can be collinear in an $n$-correct set. An $(n+1)$-node line is called a \emph{maximal line}. In $n$-correct node sets with \emph{geometric characterization}: $GC_n$ sets, introduced by Chung and Yao \cite{CY77}, the fundamental polynomial of each node is a product of $n$ linear factors.
The presence of simple fundamental polynomials, in  view of the Lagrange formula, means a simple formula for interpolation polynomials. Thus, from both theoretical and practical points of view, the $GC_n$ sets are the simplest $n$-correct sets.
Consequently, the study of the sets $ GC_n $ is a topical problem in the theory of bivariate polynomial interpolation.

The conjecture of M. Gasca and J. I. Maeztu \cite{GM82} (GM conjecture) states that every $GC_n$ set has at least one maximal line. Until now the conjecture has been
proved only for the cases $n \le 5$ (see \cite{B90} and \cite{HJZ14}). The GM conjecture is very important because if it proves to be true then a nice classification of $ GC_n $ sets takes place (see the forthcoming Sections 3-5).

We say  that a node \emph{uses} a line if the line is a factor of the fundamental polynomial of this node. The problem of finding the fundamental polynomial of a node in the $ GC_n $ sets means finding all the lines used by the node.
To solve this important problem, we can look at it more generally. Namely, given a line and a node, find out whether the line is used by the node.
Therefore, the problem under consideration is equivalent to the problem of finding all nodes that use a given line.

Denote  the subset of nodes of $\mathcal X$ that use the line $\ell$ by ${\mathcal X}^\ell.$
Let us mention that
\begin{equation}\label{max} {\mathcal X}^\ell =\mathcal
X\setminus \ell\ \hbox{if $\ell$ is a maximal line}.\end{equation}
Next, let us formulate a recent result regarding the set ${\mathcal X}^\ell$ for the general $\ell$ (see also \cite{BH}, Conjecture 3.7).

\begin{theorem}[\cite{HV2}, Theorem 3.1] \label{mainc}
 Let ${\mathcal X}$ be a $GC_n$ set, and
${\ell}$ be a $k$-node line, $k\ge  2.$ Assume that GM Conjecture holds for all degrees up to $n$. Then we have that
\begin{equation*} \label{11bbaa}{\mathcal X}^{\ell} =\emptyset,\ \hbox{or}
\end{equation*}
\begin{equation} \label{111aa} {\mathcal X}^\ell \ \hbox{is a $GC_{s-2}$  subset  of ${\mathcal X},$ hence}\  |{\mathcal X}_{\ell}| = \binom{s}{2},
\end{equation}
where $k-\delta\le s  \le k$ and $\delta=n+1-k.$
\end{theorem}
Moreover, \cite{HV2} established the following algorithm for retrieving ${\mathcal X}^\ell.$
At the first step one obtains a node set ${\mathcal X}_1$ by removing the nodes in a maximal line of ${\mathcal X}$, which does not intersect the line $\ell$ at a node of ${\mathcal X}$ or by removing the nodes in a pair of maximal lines of ${\mathcal X}$ which are concurrent together with $\ell,$ i.e.,  the two maximal lines and $\ell$ intersect at a single point. Thus, in view of the forthcoming Proposition \ref{crl:minusmax}, ${\mathcal X}_1$ is a $GC_{n-1}$ or $GC_{n-2}$ set. Second step is similar to the first one. The only difference is that instead of the maximal lines of ${\mathcal X}$ now the maximal lines of ${\mathcal X}_1$ are used. By continuing in the same way, at the last step $\omega$, the line $\ell$ becomes a maximal  line in a $GC_{s-1}$ set ${\mathcal X}_\omega,$ and hence, according to \eqref{max}, one gets
$${\mathcal X}^\ell={\mathcal X}_\omega\setminus \ell.$$

In this paper, by assuming that GM conjecture is true, we characterize the lines $\ell$ for which the above-described algorithm can be simplified and done by using only the original maximal lines of ${\mathcal X}.$

More specifically, we introduce the concept of \emph{$\ell$-lowering} of a $GC_n$ set ${\mathcal X},$ and denote it by $\hat{\mathcal X}:=\hat{\mathcal X}(\ell).$ This set plays an important role here. We get $\hat{\mathcal X}$ by removing, at a single step, the nodes in all maximal lines of ${\mathcal X}$ which do not intersect the line $\ell$ at a node of ${\mathcal X}$ and by removing  the nodes in all pairs of maximal lines of ${\mathcal X}$ which are concurrent together with $\ell.$ Then we specify cases when $\ell$ is a maximal line in the obtained $GC$ set $\hat{\mathcal X}$ and hence
$${\mathcal X}^\ell=\hat{\mathcal X}\setminus\ell.$$
We call such lines \emph{proper} lines.
We prove in the forthcoming Theorem \ref{th:xl} that all lines $\ell$ used by more than $3$ nodes are either maximal or proper, provided that the number of maximal lines of ${\mathcal X}$ differs from three.
Then we prove that in the case of arbitrary non-proper line $\ell$ one needs to complete only one or two steps of the  maximal lines removal in $\hat{\mathcal X}$ till $\ell$ becomes a maximal line. It means that in the general case the above-described algorithm can be carried out by using at most only two steps of the maximal lines removal where the lines
are not original maximal lines of  ${\mathcal X}.$

It is worth mentioning that if the number of maximal lines of ${\mathcal X}$ equals three, in which case ${\mathcal X}$ is called \emph{generalized principal lattice}, we have a simple description of the set ${\mathcal X}^\ell$ (see  the forthcoming formula \eqref{newton}).

Then, in contrast to the ambiguity present in the formula \eqref{111aa}, we determine the exact number of the nodes that use a given line, provided that GM conjecture is true. Let us call a node a \emph{$2_m$-node} in ${\mathcal X}$ if it is a point of intersection of two maximal lines. In the forthcoming Theorem \ref{th:2} we prove that a $k$-node line $\ell$ either is not used at all, or is used by exactly $\binom{k-r-\hat r}{2}$ nodes of $\mathcal X,$ where $r$ is the number of  $2_m$-nodes in $\ell\cap{\mathcal X}$ and $\hat r$ is the number of $2_m$-nodes in $\ell\cap\hat{\mathcal X}.$
Moreover, we have that always $\hat r\le 2.$ On other hand we have that $\hat r =0$ if $\#{\mathcal X}^\ell > 3.$
Furthermore, for each $2_m$-node  in $\ell\cap\hat{\mathcal X}$ we prove that one of the two maximal lines in $\hat{\mathcal X}$ to which it belongs is a maximal line in ${\mathcal X}$ and another is a proper line in ${\mathcal X}.$

Next, we study the set of all $n$-node lines of $GC_n$ set ${\mathcal X}$ denoted by $N({\mathcal X}).$

Also, in this paper we improve several known properties of $GC_n$ sets.

Let us mention that Carnicer and Gasca started the investigation of the set ${\mathcal X}^\ell$ and proved that a $k$-node line $\ell$ can be used by at most $\binom{k}{2}$ nodes of a $GC_n$ set $\mathcal X$ and in addition there are no $k$ collinear nodes that use $\ell$, provided that GM conjecture is true (see \cite{CG03}, Theorem 4.5). Note that these statements follow readily from Theorem \ref{mainc}.


\section{$GC_n$ sets and the Gasca-Maeztu conjecture \label{ss:GMconj}}

Let $\Pi_n$ be the space of bivariate polynomials of total degree
at most $n:$
\begin{equation*}
p(x,y)=\sum_{i+j\leq{n}}c_{ij}x^iy^j.
\end{equation*}
We have that $N:=\dim \Pi_n=\binom{n+2}{2}.$\\
Let ${\mathcal X}$ be a set of $N$ distinct nodes: ${\mathcal X}=\{ (x_1, y_1), (x_2, y_2), \dots , (x_N, y_N) \} .$

\begin{definition}
A set of nodes ${\mathcal X},\ \#{\mathcal X}=N,$ is called \emph{$n$-correct} if
for any given numbers $c_1 c_2 \dots c_N$ there exists a unique polynomial
$p \in \Pi_n$ satisfying the following conditions
\begin{equation}\label{int cond}
p(x_i, y_i) = c_i, \ \ \quad i = 1, 2, \dots, N.
\end{equation}
\end{definition}

A polynomial $p\in\Pi_n$ is called an
\emph{$n$-fundamental polynomial} of a node
$A\in{\mathcal X},$  if
$ p(A) =1\  \text{and}\  p(B) = 0\  \forall\ B\in{\mathcal X}\setminus\{A\}.$ We denote this polynomial by  $p_{A,{\mathcal X}}^\star.$

\begin{definition}
Let ${\mathcal X}$ be an $n$-correct set. We say that a node
$A\in{\mathcal X}$ \emph{uses} a line $\ell\in \Pi_1,$ if
$p_{A,{\mathcal X}}^\star = \ell q, \ \text{where} \ q\in\Pi_{n-1}.
$
\end{definition}
The following proposition is well-known (see, e.g., \cite{HJZ14}
Proposition 1.5):
\begin{proposition}\label{prp:n+1points}
Suppose that a polynomial $p \in
\Pi_n$ vanishes at $n+1$ points of a line $\ell.$ Then we have that
$p = \ell r, \ \text{where} \ r\in\Pi_{n-1}.$
\end{proposition}

\noindent Thus at most $n+1$ nodes of an $n$-correct set ${\mathcal X}$ can be collinear.
A line $\lambda$ passing through $n+1$
nodes of the set ${\mathcal X}$ is called a \emph{maximal line}.
Note that from Proposition \ref{prp:n+1points} we readily get \eqref{max}.

Below the basic properties of maximal lines are given:
\begin{proposition}[\cite{CG00}, Prop. 2.1]\label{properties}
Let ${\mathcal X}$ be an $n$-correct set. Then
\vspace{-1.5mm} \begin{enumerate} \setlength{\itemsep}{0mm}
\item
any two maximal lines of ${\mathcal X}$ intersect necessarily at a node of ${\mathcal X};$
\item
any three maximal lines of ${\mathcal X}$ cannot be concurrent;
\item
${\mathcal X}$ can have at most $n+2$ maximal lines.
\end{enumerate}
\end{proposition}


Now let us consider a special type of $n$-correct sets satisfying a geometric characterization (GC) property introduced by K.C. Chung and T.H. Yao:
\begin{definition}[\cite{CY77}]
An $n$-correct set ${\mathcal X}$ is called a \emph{$GC_n$ set (or $GC$ set)}
 if  the
$n$-fundamental polynomial of each node $A\in{\mathcal X}$ is a
product of $n$  linear factors.
\end{definition}
Thus, each node of a $GC_n$ set uses exactly $n$ lines.

\begin{proposition}[\cite{CG03}, Prop. 2.3] \label{crl:minusmax}
Let $\lambda$ be a maximal line of  a $GC_n$ set ${\mathcal X}.$  Then the set  ${\mathcal X}\setminus \lambda$ is a $GC_{n-1}$ set. Also any maximal line of ${\mathcal X}$ different from $\lambda$ is a maximal line of  ${\mathcal X}\setminus \lambda.$
\end{proposition}

Next we present the Gasca-Maeztu (GM) conjecture:

\begin{conjecture}[\cite{GM82}, Sect. 5]\label{conj:GM}
For any $GC_n$ set there exists at least
one maximal line.
\end{conjecture}

\noindent  Till now, this conjecture has been confirmed for the degrees
$n\leq 5$ (see \cite{B90}, \cite{HJZ14}).
For a generalization of the Gasca-Maeztu conjecture to maximal curves see \cite{HR}.

Denote by  $M({\mathcal X})$ the set of maximal lines of the node set  ${\mathcal X}.$

The concept of the defect introduced by Carnicer and Gasca in \cite{CG00} is an important characteristic of $GC_n$ sets.
\begin{definition}[\cite{CG00}]  The \emph{defect} of an $n$-correct set ${\mathcal X}$ is the number $\hbox{def}({\mathcal X}):=n+2-\#M({\mathcal X}).$
\end{definition}
In view of Proposition \ref{properties} we have that $0\le \hbox{def}({\mathcal X})\le n+2.$

\begin{proposition}[\cite{CG03}, Crl. 3.5]\label{prp:CG-1} Let $\lambda$ be a maximal line of a $GC_n$ set ${\mathcal X}$ such that $\#M({\mathcal X}\setminus \lambda)\ge 3.$ Then we have that $$\hbox{def}({\mathcal X}\setminus \lambda)=\hbox{def}({\mathcal X}) \quad \hbox{or}\quad  \hbox{def}({\mathcal X})-1.$$
\end{proposition}
This equality means that
$\#M({\mathcal X}\setminus \lambda)=\#M({\mathcal X})-1\ \ \hbox{or}\ \ \#M({\mathcal X}).$\\
In view of Proposition \ref{crl:minusmax} all $\#M({\mathcal X})-1$ maximal lines of ${\mathcal X}$ different from $\lambda$ belong to $M({\mathcal X}\setminus \lambda).$ Thus there can be at most one \emph{newly emerged} maximal line of ${\mathcal X}\setminus \lambda.$

In the sequel we will use the following
\begin{lemma}[\cite{CGo06}, Lemma 3.4] \label{pluslem}Suppose that the Gasca-Maeztu conjecture is true for
all degrees up to $n.$ Suppose also that ${\mathcal X}$ is a $GC_n$ set with exactly three maximal lines and $\lambda$
is a maximal line. Then, the $GC_{n-1}$ set ${\mathcal X}\setminus \lambda$ also possesses exactly three maximal
lines, hence $\hbox{defect}({\mathcal X}\setminus\lambda)=\hbox{defect}({\mathcal X})-1.$
\end{lemma}

\begin{definition}[\cite{CG01}]  Given an $n$-correct set ${\mathcal X}$ and a  line $\ell,$
${\mathcal X}^\ell$ is the subset of nodes of
${\mathcal X}$ which use the line $\ell.$
\end{definition}

Let ${\mathcal X}$ be an $n$-correct set, and ${\ell}$ be a line. Then\\
(i) a maximal line $\lambda$ is called $\ell$-\emph{disjoint} if
$\lambda \cap {\ell} \cap  {\mathcal X} =\emptyset;$\\
(ii) two maximal lines $\lambda', \lambda''$ are called $\ell$-\emph{adjoint} if
$\lambda' \cap \lambda''\cap {\ell} \in {\mathcal X}.$

In view of  Proposition \ref{properties} a maximal line $\ell$ has no $\ell$-disjoint or $\ell$-adjoint maximal lines.

The following two lemmas of Carnicer and Gasca play an important role in the sequel.
\begin{lemma}[\cite{CG03}, Lemma 4.4]\label{lem:CG1}
Let ${\mathcal X}$ be an $n$-correct set and ${\ell}$ be a line with $\#(\ell\cap{\mathcal X})\le n.$ Suppose also
that a maximal line $\lambda$ is $\ell$-disjoint.
Then we have that
${\mathcal X}^{\ell} = {({\mathcal X} \setminus \lambda)}^{\ell}.$
\end{lemma}
The set ${\mathcal X} \setminus \lambda$ is called \emph{$\ell$-disjoint reduction} of ${\mathcal X}.$

\begin{lemma}[\cite{CG03}, Proof of Thm. 4.5]\label{lem:CG2}
Let ${\mathcal X}$ be an $n$-correct set and ${\ell}$ be a line with $\#(\ell\cap{\mathcal X})\le n.$ Suppose also
that two maximal lines $\lambda', \lambda''$ are $\ell$-adjoint.
Then we have that
${\mathcal X}^{\ell} = {[{\mathcal X} \setminus (\lambda' \cup \lambda'')]}^{\ell}.$
\end{lemma}
The set ${\mathcal X} \setminus (\lambda' \cup \lambda'')$ is called \emph{$\ell$-adjoint reduction} of ${\mathcal X}.$

Next, by the motivation of the above two lemmas, we introduce the concept of \emph{$\ell$-lowering} of a $GC_n$ set.
\begin{definition}\label{def:reduct} Let ${\mathcal X}$ be a $GC_n$ set, $\ell$ be a $k$-node line, $2\le k\le n+1.$ We say that the set $\hat{\mathcal X}=\hat{\mathcal X}(\ell)$ is the $\ell$-lowering  of ${\mathcal X},$ and briefly denote this by ${\mathcal X}\downarrow_\ell\hat{\mathcal X},$ if
$$\hat{\mathcal X}={\mathcal X} \setminus \left({\mathcal U}_1\cup {\mathcal U}_2\right),$$
where
${\mathcal U}_1$ is the union of the $\ell$-disjoint  maximal lines of ${\mathcal X},$ and ${\mathcal U}_2$ is the union of the  (pairs of) $\ell$-adjoint maximal lines of ${\mathcal X}.$
\end{definition}
If $\ell$ is a maximal line then ${\mathcal U}_1= {\mathcal U}_2=\emptyset$ and hence $\hat{\mathcal X}={\mathcal X}.$

From Lemmas \ref{lem:CG1} and \ref{lem:CG2} we  immediately get that
\begin{equation*}\label{abcd}
{\mathcal X}\downarrow_\ell\hat{\mathcal X} \Rightarrow {\mathcal X}^\ell=\hat{\mathcal X}^\ell.\end{equation*}
\begin{definition} \label{1m2m}  A node $A\in{\mathcal X}$  is called \emph{$k_m$-node} if it belongs to exactly $k$ maximal lines.
\end{definition}
\noindent In view of Proposition \ref{properties}, (ii),
there are  $0_m,\ 1_m,$ and $2_m$-nodes only.
It is easily seen that $\ell\cap\hat{\mathcal X}(\ell)$ contains no  $2_m$-node of ${\mathcal X}.$

The following proposition will be used frequently in the sequel.
\begin{proposition}\label{abab}
Suppose that a line $\ell\notin M({\mathcal X})$ passes through exactly $k$  $1_m$-nodes of a $GC_n$ set ${\mathcal X}.$
Then the set $\hat{\mathcal X}(\ell)$ is a $GC_s$ set with $s=\hbox{def}({\mathcal X})+k-2.$
\end{proposition}
\begin{proof} Suppose that a line $\ell$ passes through exactly $k$  $1_m$-nodes. Then there are exactly $\#M({\mathcal X})-k$ maximal lines which are either $\ell$-disjoint or $\ell$-adjoint. Hence, in view of Proposition \ref{crl:minusmax}, we readily get that
the set $\hat{\mathcal X}(\ell)$ is a $GC_s$ set with $s=n-[\#M({\mathcal X})-k]=\hbox{def}({\mathcal X})+k-2.$
\end{proof}
\begin{definition}\label{def:reductcor} Let ${\mathcal X}$ be a $GC_n$ set, $\ell$ be a $k$-node line, $2\le k\le n$ and ${\mathcal X}\downarrow_\ell\hat{\mathcal X}.$
Then the line $\ell$ is called \emph{proper} if it is a maximal line in the set $\hat{\mathcal X}.$ \\
The line $\ell$ is called \emph{proper} $(-r)$ if it becomes a maximal line after $r$ steps of application of $\ell$-disjoint or $\ell$-adjoint reductions, in all, starting with ${\hat{\mathcal X}}.$ \\
We call the $\ell$-disjoint maximal line, or the union of the pair of $\ell$-adjoint maximal lines, used at the above-mentioned $k$th step, the $(-k)$ \emph{d/a (disjoint/adjoint)} item,   $k=1,\dots,r.$
\end{definition}
In the forthcoming Theorem \ref{th:xl} we show that if $\hbox{def}({\mathcal X})\neq n-1$ then any used line, which is not maximal, is either proper, proper $(-1)$, or $(-2)$.

Denote by  $Pr({\mathcal X})$ the set of proper lines of ${\mathcal X}.$\\
From \eqref{max} we immediately get that $ {\mathcal X}^\ell=\hat{\mathcal X} \setminus \ell$ if $\ell\in Pr({\mathcal X}).$

In Proposition \ref{pr} we
show that $\# Pr({\mathcal X})\in\{0,3\}$ if $\hbox{def}({\mathcal X})\neq 1,$ provided that the Gasca-Maeztu conjecture is true.

In the next three sections we will consider  the results of Carnicer, Gasca, and God\'es, concerning the classification of $GC_n$ sets according to the defects of the sets.
Let us present here the following
\begin{theorem}[\cite{CGo10}]\label{th:CGo10} Let ${\mathcal X}$ be a $GC_n$ set. Assume that the $GM$ Conjecture holds for all degrees up to $n$. Then $\hbox{def} ({\mathcal X})\in\left\{0,1,2,3, n-1\right\}.$
\end{theorem}
Of course this implies that $\#M({\mathcal X})\in\left\{3, n-1, n,n+1,n+2\right\}.$



\section{Sets of defect $0$ and $1$}

Let us start with defect $0$ sets, i.e., the Chung-Yao lattices.
Let a set ${\mathcal M}$ of $n+2$ lines be in general position, i.e., no two lines are parallel and no three lines are concurrent. Then the Chung-Yao lattice is defined as
the set ${\mathcal X}$ of all $\binom{n+2}{2}$ intersection points of these lines (see Fig. \ref{ChY}). We have that the $n+2$ lines of ${\mathcal M}$ are maximal lines for ${\mathcal X}.$ Any particular node
here is $2_m$-node, i.e., lies in exactly $2$ maximal lines. Observe that the product of the remaining $n$ maximal lines gives
the fundamental polynomial of the particular node. Thus ${\mathcal X}$ is a
$GC_n$ set.
Let us mention that any $n$-correct set ${\mathcal X},$ with $\hbox{def}({\mathcal X})=0,$ i.e., $\#M(\mathcal X)=n+2,$ in view of Proposition \ref{properties}, (i) and (ii),  forms a Chung-Yao lattice. Recall that there are no $n$-correct sets with more maximal lines (Proposition \ref{properties}, (iii)).

Evidently the set of used lines in the Chung-Yao lattice coincides with the set of $n+2$ maximal lines. For each maximal line $\ell$, in view of \eqref{max}, we have that ${\mathcal X}^\ell={\mathcal X}\setminus \ell,$ hence $\#{\mathcal X}^\ell =\binom{n+1}{2}.$

Thus the total number of line-usages equals:

$(n+2)\binom{n+1}{2} = \frac{1}{2}(n+2)(n+1)n=n\binom{n+2}{2}.$

\begin{figure}
\begin{center}
\includegraphics[width=7.0cm,height=2.cm]{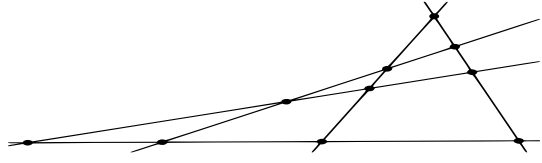}
\end{center}
\caption{A Chang-Yao lattice, $n=3.$} \label{ChY}
\end{figure}

Next let us discuss defect $1$ sets, i.e., the Carnicer-Gasca lattices.
Let a set ${\mathcal M}$ of $n+1$ lines be in general position, $n\ge 2.$ Then
the Carnicer-Gasca lattice ${\mathcal X}$ is defined as ${\mathcal X}:={\mathcal X}^{(2)}\cup{\mathcal X}^{(1)},$
where ${\mathcal X}^{(2)}$ is the set of all intersection nodes of these $n+1$ lines, and ${\mathcal X}^{(1)}$ is a set of other $n+1$  non-collinear nodes,
one in each line, to make the line maximal (see Fig. \ref{cg}). We have that $\#{\mathcal X}=\binom{n+1}{2}+(n+1)=\binom{n+2}{2}.$
It is easily seen that ${\mathcal X}$  is a
$GC_n$ set and has exactly $n+1$ maximal lines, i.e., the lines of ${\mathcal M}.$ The set ${\mathcal X}^{(2)}$ consists of  $2_m$-nodes and the set ${\mathcal X}^{(1)}$ consists of $1_m$-nodes.
Let us mention that any $n$-correct set ${\mathcal X},$ with $\hbox{def}({\mathcal X})=1,$ i.e., $\#M(\mathcal X)=n+1,$ forms a Carnicer-Gasca lattice (see \cite{CG00}, Proposition 2.4).


It is easily seen that the set of used lines in the Carnicer-Gasca lattice consists of two classes:

1) The set of $n+1$ maximal lines;

2) The set of lines passing through at least two $1_m$-nodes.

Next we show that the lines of the class 2) are proper:

\begin{proposition} \label{prp:def1}
Suppose that ${\mathcal X}$ is a Carnicer-Gasca lattice and $\ell$ is a line passing through exactly $k$  $1_m$-nodes, where $k\ge 2.$
Suppose also that ${\mathcal X}\downarrow_\ell \hat{\mathcal X}$. Then $\hat{\mathcal X}$ is a Chung-Yao lattice of degree $k-1,$ where $\ell$ is a maximal line. Hence ${\mathcal X}_\ell=\hat{\mathcal X}\setminus \ell$ and
$\#{\mathcal X}^\ell = \binom{k}{2}.$
\end{proposition}

\begin{proof}  In view of Proposition \ref{abab} we have that $\hat{\mathcal X}$ is a $GC_{k-1}$ set. On the other hand we have that
the line $\ell,$ with $k$  $1_m$-nodes, as well as the $k$ maximal lines of ${\mathcal X}$ intersecting the line $\ell$ at these nodes, in all $k+1$ lines, are maximal lines in $\hat{\mathcal X}.$
\end{proof}

\begin{figure}
\begin{center}
\includegraphics[width=7.0cm,height=2.cm]{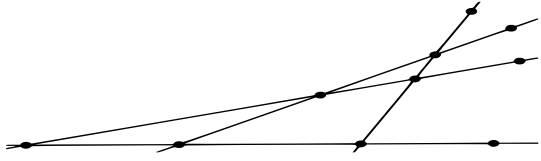}
\end{center}
\caption{A Carnicer-Gasca lattice, $n=3.$}\label{cg}
\end{figure}

Denote by $n_i$ the number of line usages in the class $i),\  i=1,2.$ Then we have that the total number of line-usages equals:

$n_1+n_2=(n+1)\binom{n+1}{2}+ \binom{n+1}{2} = \frac{1}{2}(n+1)(n+1)n+\frac{1}{2}(n+1)n=n\binom{n+2}{2}.$

\section{Defect $2$ sets}

Let a set ${\mathcal M}$ of $n$ lines be in general position, $n\ge 3.$ Then
consider the lattice ${\mathcal X}$ defined as
\begin{equation}\label{01O}{\mathcal X}:={\mathcal X}^{(2)}\cup{\mathcal X}^{(1)}\cup{\mathcal X}^{(0)},\end{equation}
where ${\mathcal X}^{(2)}$ is the set of all intersection nodes of these $n$ lines, ${\mathcal X}^{(1)}$ is a set of other $2n$ nodes,
two in each line, to make the line maximal and ${\mathcal X}^{(0)}$ consists of a single node, denoted by $O,$ which does not belong to any line from ${\mathcal M}$ (see Fig. \ref{def2}).
Correspondingly, we have that $\#{\mathcal X}=\binom{n}{2}+2n+1=\binom{n+2}{2}.$

Note that all the nodes of ${\mathcal X}^{(k)}$ belong to exactly $k$ maximal lines and thus are $k_m$-nodes, $k=0, 1, 2.$

In the sequel we will need the following characterization of $GC_n$ set ${\mathcal X},$ with $\hbox{def}(\mathcal X)=2,$  due to Carnicer and Gasca:

\begin{proposition}[\cite{CG00}, Prop. 2.5]\label{prp:nmax}
A set ${\mathcal X}$ is a $GC_n$ set of defect $2,$ i.e., has exactly $n$ maximal lines: $\lambda_1,\ldots,\lambda_n,$ where $n\ge 3,$ if and only if, \eqref{01O} holds with the following additional properties  (see Fig. \ref{def2}):

\vspace{-1.5mm} \begin{enumerate} \setlength{\itemsep}{0mm}
\item  There are $3$ lines $\ell_1^o,\ell_2^o,\ell_3^o,$ called $O$-lines, concurrent at  $O:\\ O=\ell_1^o\cap \ell_2 ^o\cap \ell_3^o,$  such  that  ${\mathcal X}^{(1)}\subset \ell_1^o\cup \ell_2^o \cup \ell_3^o;$
\item No line $\ell_i^o$ contains $n+1$ nodes of ${\mathcal X},\ i=1,2,3.$
\end{enumerate}
\end{proposition}
Note that each above line $\ell_i^o$ contains at least two $1_m$-nodes. They may contain also $2_m$-nodes.

\subsection{The used lines in defect $2$ sets}

Suppose that $M({\mathcal X})=\{\lambda_1,\ldots,\lambda_n\}.$ Consider a pair of maximal lines $\lambda_i, \lambda_j,$ $1\le i<j\le n.$ We have that the node $A_{ij}:=\lambda_i\cap \lambda_j$ uses $n-2$ maximal lines, i.e., all maximal lines except  $\lambda_i, \lambda_j.$ Next let us identify the remaining two used lines.
We have two $1_m$-nodes in each of the two maximal lines, which, in view of the above item (i), lie also in the three $O$-lines. Denote by $n_{ij}$ the number of $O$-lines containing at least one of these four $1_m$-nodes. Evidently we have that  $n_{ij}= 2$ or $3.$ In the case $n_{ij}= 2$  (the case of the node $A_{34}$ in Fig. \ref{def2}) the four $1_m$-nodes belong to two $O$-lines, which are the two remaining used lines. In the case  $n_{ij}=3$ (the case of the node $A_{12}$ in Fig. \ref{def2}) two of four $1_m$-nodes belong to the same $O$-line, denoted by $\ell_{k(i,j)}^o,$ where $1\le k(i,j)\le 3.$ Denote also by $\ell_{ij}$ the line passing through the other two $1_m$-nodes.  It is easily seen that in this case the lines $\ell_{k(i,j)}^o$ and $\ell_{ij}$ are the remaining two lines used by the node $A_{ij}.$
\begin{figure}\label{def2}
\begin{center}
\includegraphics[width=9.0cm,height=3.cm]{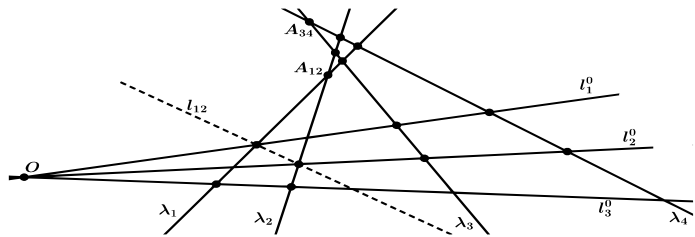}
\end{center}
\caption{A defect 2 set, $n=4.$}
\end{figure}

The set of used lines in ${\mathcal X}$ consists of the following three classes:

1) The set of $n$ maximal lines;

2) The set of three $O$-lines $\{\ell_1^o, \ell_2^o,\ell_3^o\};$

3) The set of lines $\ell_{ij},\ 1\le i<j\le n,$ with $n_{ij}=3.$

Let us verify that these classes are disjoint, i.e., lines from different classes cannot coincide.
It suffices to show that the classes 2) and 3) are disjoint. Indeed, any line $\ell_{ij}$ passes through two $1_m$-nodes which belong to two different $O$-lines. Hence  $\ell_{ij}$ cannot coincide with an $O$-line.

Note also that $\ell_{ij}=\ell_{i'j'}$ implies that $(i,j)=(i',j').$ Indeed, there is an $O$-line which intersects both lines at nodes of ${\mathcal X}.$ Hence without loss of generality we may assume that $\lambda_i=\lambda_i',$ i.e., $i=i'.$ Now suppose by way of contradiction that $j\neq j'.$ Then the other two $O$-lines do not intersect the maximal line $\lambda_i$ at nodes of ${\mathcal X}$, which contradicts   Proposition \ref{prp:nmax}, (i).

Next let us show that the lines in the class 2) are proper:

\begin{proposition} \label{prp1:def2}
Suppose that ${\mathcal X}$  is a $GC_n$ set of defect $2$ and the line $\ell:=\ell_i^o,\ i\in\{1,2,3\},$ passes  through exactly $k$  $1_m$-nodes, where $k\ge 2.$
Suppose also that ${\mathcal X}\downarrow_\ell \hat{\mathcal X}$. Then $\hat{\mathcal X}$ is a $GC_k$ set of defect $1$, where $\ell$ is a maximal line. The set ${\mathcal X}^\ell=\hat{\mathcal X}\setminus \ell$ is a $GC_{k-1}$ set of defect $1$.
\end{proposition}
\begin{proof}  In view of Proposition \ref{crl:minusmax}, we have that $\hat{\mathcal X}$ is a $GC_{k}$ set.
Then note that the line $\ell$ is a $(k+1)$-node line  in $\hat{\mathcal X},$ with the node $O$ and $k$ $1_m$-nodes.
Hence the line $\ell,$ as well as the $k$ maximal lines of ${\mathcal X}$ intersecting the line $\ell$ at $k$  $1_m$-nodes are maximal lines in $\hat{\mathcal X}.$ Now assume by way of contradiction that there is one more maximal line in $\hat{\mathcal X}.$ Then clearly it coincides with an $O$ line $\ell_j^o, j\neq i.$ Observe that each disjoint or adjoint reduction in getting $\hat{\mathcal X}$ from ${\mathcal X}$ takes a node or two nodes from $\ell_j^o,$ respectively. Therefore we obtain that $\ell_j^o$ is a maximal line in ${\mathcal X},$ which contradicts Proposition \ref{prp:nmax}, (ii).
\end{proof}

\begin{definition} \label{special}
Suppose $\ell$ is the line which is not maximal or proper. A node $A\in\ell$ is called a $\hat 2_m$-node if it is a $2_m$-node in the set $\hat{\mathcal X}=\hat{\mathcal X}(\ell).$
\end{definition}
It can be readily verified that  in Fig. \ref{Sde55} $S$ is a $\hat 2_m$-node in the line $\ell_{12}.$

\noindent In Proposition \ref{th:xl} we will prove that if $S\in\ell$ is a $\hat 2_m$-node then

$S=\ell^o\cap\lambda,\ \hbox{where}\ \ell^o\in Pr({\mathcal X})\cap M(\hat{\mathcal X})\ \hbox{and}\ \lambda\in M({\mathcal X}).$

Thus, any $\hat 2_m$-node is necessarily an $1_m$-node for ${\mathcal X}.$   We also will verify  that each used line $\ell$ may have at most two $\hat 2_m$-nodes.

Next let us show that the lines in the class 3) are proper $(-1)$ and are used by only one node.

\begin{proposition} \label{prp2:def2}
Suppose that ${\mathcal X}$ is a $GC_n$ set of defect $2,\ \ell:=\ell_{ij},$ with $n_{ij}=3,\ 1\le i<j\le n,$  and ${\mathcal X}\downarrow_\ell \hat{\mathcal X}.$ Then the point $S:=\ell_{k(i,j)}^o\cap\ell$ is the only candidate for a $\hat 2_m$-node in $\ell$ and
\begin{equation*}\label{s}S\in \hat{\mathcal X}\Leftrightarrow S\ \hbox{is an $1_m$-node in}\ {\mathcal X}\Leftrightarrow  S\ \hbox{is a $\hat 2_m$-node}\ \Leftrightarrow S=\ell^o\cap\lambda^*,\end{equation*}
where $\ell^o:=\ell_{k(i,j)}^o\in\Pr({\mathcal X})\cap M(\hat{\mathcal X})$ and $\lambda^*\in M({\mathcal X}).$

\noindent Next we have that $\hbox{def}(\hat{\mathcal X})=1,\ \#{\mathcal X}^\ell = 1,$ and the line $\ell$ is proper $(-1),$   i.e.,
$${\mathcal X}^\ell = \hat{\mathcal X}\setminus ({\mathcal C}\cup \ell),\ \hbox{where ${\mathcal C}$ is the d/a item.}$$
Moreover if $S\notin \hat{\mathcal X}$ then $\hat{\mathcal X}$ is a $GC_2$ set,  ${\mathcal C}=\ell_{k(i,j)}^o\in M(\hat{\mathcal X}).$
While if $S\in \hat{\mathcal X},$  then $\hat{\mathcal X}$ is a $GC_3$ set, ${\mathcal C}=\ell_{k(i,j)}^o\cup \lambda^*,$ where $\lambda^*\in M({\mathcal X})$ and $S\in\lambda^*.$
\end{proposition}
\begin{proof}  Note that, in view of Proposition \ref{prp:nmax}, (i), all $1_m$-nodes of ${\mathcal X}$ belong to the three $O$-lines. Assume first that $S\notin \hat{\mathcal X}$ (see the case of $\ell=\ell_{12},$ and $k(i,j)=3,$ in Fig. \ref{def2}). Then the line $\ell$ intersects only two $O$-lines in the nodes of ${\mathcal X}.$ Hence it has just two $1_m$-nodes. Thus, in view of Proposition \ref{abab}, we have that $\hat{\mathcal X}$ is a $GC_{2}$ set. It is easily seen that $M(\hat{\mathcal X})=\{\lambda_i,\lambda_j,\ell_{k(i,j)}^o\},$ hence $\hbox{def}(\hat{\mathcal X})=1.$ Also, $\ell$ is a $2$-node maximal line in $\hat{\mathcal X}\setminus \ell_{k(i,j)}^o$ and $\#{\mathcal X}^\ell = 1.$ Clearly $\ell$ has no $\hat 2_m$-node.

\begin{figure}
\begin{center}
\includegraphics[width=9.0cm,height=3.cm]{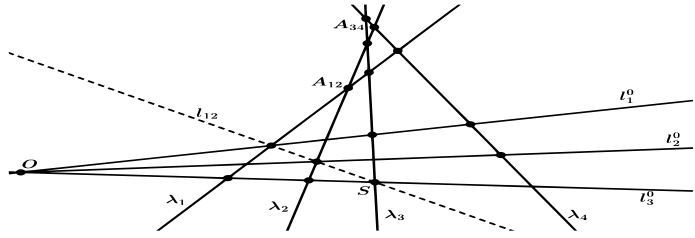}
\end{center}
\caption{$S$ is a $\hat 2_m$-node in $\ell_{12}.$}\label{Sde55}
\end{figure}

Now assume that $S\in \hat{\mathcal X}$ (see the case of $\ell=\ell_{12}, \lambda^\star=\lambda_3,$ and $k(i,j)=3,$ in Fig. \ref{Sde55}). Then, since $S\neq O$ and $S$ is not $2_m$-node in ${\mathcal X},$ we conclude that $S$ is an $1_m$-node in ${\mathcal X}.$ Thus $S\in\lambda^*,$ where $\lambda^*\in M({\mathcal X}).$ Therefore the line $\ell$ has exactly three $1_m$-nodes. Next, in view of Proposition \ref{abab}, we have that $\hat{\mathcal X}$ is a $GC_{3}$ set. In this case we have $M(\hat{\mathcal X})=\{\lambda_i,\lambda_j,\lambda^*,\ell_{k(i,j)}^o\}$ and hence $\hbox{def}(\hat{\mathcal X})=1.$ It is easily seen that $\ell$ is a $3$-node line in $\hat{\mathcal X}$ with only one $\hat 2_m$-node: $S=\ell_{k(i,j)}^o\cap\lambda^*.$
Moreover, we have that $\ell$ is a maximal line in $\hat{\mathcal X}\setminus (\ell_{k(i,j)}^o\cup\lambda^*)$ and $\#{\mathcal X}^\ell = 1.$
\end{proof}

Denote by $n_i$  the number of line usages in the class $i),\ i=1,2,3.$ Then we have that the total number of line-usages equals:

$n_1+[n_2+n_3]=n\binom{n+1}{2}+ [2\binom{n}{2}+2n] = \frac{1}{2}n(n+1)n+n(n-1)+2n=n\binom{n+2}{2}.$

Here we take into account the fact that each $2_m$-node uses exactly two lines from the classes 2) and 3), while each of $2n$ $1_m$-nodes uses only one line from there, namely from the class 3).

\section{Defect $3$ sets}

Let a set ${\mathcal M}=\{ \lambda_1,\ldots,\lambda_{n-1}\}$ of $n-1$ lines be in general position, $n\ge 4.$ Then
consider the lattice ${\mathcal X}$ defined as
\begin{equation}\label{01OOO}{\mathcal X}:={\mathcal X}^{(2)}\cup{\mathcal X}^{(1)}\cup{\mathcal X}^{(0)},\end{equation}
where ${\mathcal X}^{(2)}$ is the set of all intersection nodes of these $n-1$ lines, ${\mathcal X}^{(1)}$ is a set of other $3(n-1)$ nodes,
three in each line, to make the line maximal and ${\mathcal X}^{(0)}$ consists of exactly three non-collinear nodes, denoted by $O_1,O_2,O_3,$ which do not belong to any line from ${\mathcal M}$ (see Fig. \ref{Def3})
Correspondingly, we have that $\#{\mathcal X}=\binom{n-1}{2}+3(n-1)+3=\binom{n+2}{2}.$

Note that in the end all the nodes of ${\mathcal X}^{(k)}$ will belong to exactly $k$ maximal lines and thus become $k_m$-nodes, $k=0, 1, 2.$

Denote by $\ell_{i}^{oo},\ 1\le i\le 3,$ the line passing through the two $0_m$-nodes in $\{O_1,O_2,O_3\}\setminus \{O_i\}.$ We call these lines $OO$-lines. Suppose that
${\mathcal X}^{(1)}=\{A_i^1, A_i^2, A_i^3\in \lambda_i:\ 1\le i\le n-1\}.$

In the sequel we will need the following characterization of $GC_n$ set ${\mathcal X},$ with $\hbox{def}(\mathcal X)=3,$ due to Carnicer and God\'es:

\begin{proposition}[\cite{CGo07}, Thm. 3.2]\label{prp:n-1max}
A set ${\mathcal X}$ is a $GC_n$ set of defect $3,$ i.e., has exactly $n-1$ maximal lines $\lambda_1,\ldots,\lambda_{n-1},$ where $n\ge 4,$ if and only if, with some permutation of the indices of the maximal lines and $1_m$-nodes, \eqref{01OOO} holds with the following additional properties (see Fig. \ref{Def3}):
\vspace{-1.5mm} \begin{enumerate} \setlength{\itemsep}{0mm}
\item  ${\mathcal X}^{(1)}\setminus (\ell_{1}^{oo}\cup \ell_{2}^{oo} \cup \ell_{3}^{oo})=\{ D_{1}, D_2, D_3\},$ where $D_i:=A_i^i;$
\item Each line $\ell_{i}^{oo}, i=1,2,3,$ passes through exactly $n$ nodes: $n-2$  $1_m$-nodes and two $0_m$-nodes. Moreover, $\ell_{i}^{oo}\cap\lambda_i\notin{\mathcal X},\ i=1,2,3;$
\item  The triples $\{O_1,D_2,D_3\},\ \{O_2,D_1,D_3\},\ \{O_3,D_1,D_2\}$ are collinear.
\end{enumerate}
\end{proposition}
Let us denote by $\ell_{1}^{dd}, \ell_{2}^{dd}, \ell_{3}^{dd},$ the lines passing through the latter triples, respectively, and call them \emph{$DD$-lines.} Also, the nodes $D_i$ are called \emph{$D$-nodes.}
\begin{figure}
\begin{center}
\includegraphics[width=11.0cm,height=3.5cm]{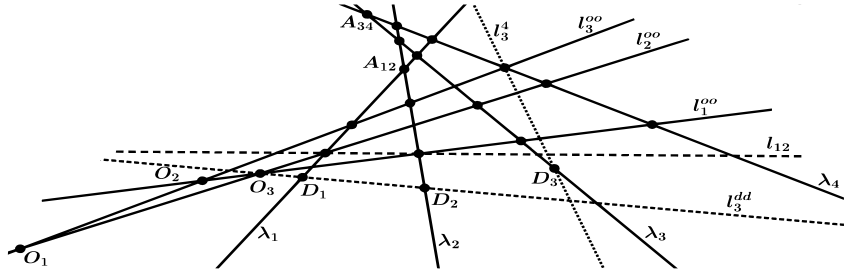}
\end{center}
\caption{A defect 3 set, $n=5.$}\label{Def3}
\end{figure}

\subsection{The used lines in defect $3$ sets}

Let $M({\mathcal X})=\{\lambda_1,\ldots,\lambda_{n-1}\}.$ Consider a pair of maximal lines $\lambda_i, \lambda_j,$ where $1\le i<j\le 3.$ Assume that $\{i,j,k\}=\{1,2,3\}.$ We have that the node $A_{ij}:=\lambda_i\cap \lambda_j$ uses $n-3$ maximal lines, i.e., all the maximal lines except  $\lambda_i$ and $\lambda_j$  (see the case of the node $A_{12}, k=3,$ in Fig. \ref{Def3}). Next let us identify the remaining three used lines.
We have three $1_m$-nodes in each of two maximal lines, six in all.  Two of these nodes  belong to the first used line: $\ell_k^{oo}.$ Two nodes belong to the second used line:  $\ell_k^{dd}.$
Finally, the last used line is denoted by $\ell_{ij},$ which passes through the remaining two $1_m$-nodes in the lines $\lambda_i$ and $\lambda_j.$

Now consider a pair of maximal lines $\lambda_i, \lambda_j,$ where $1\le i\le 3$ and $4\le j\le n-1.$ Assume that $\{i,k_1,k_2\}=\{1,2,3\}.$ As above the node $A_{ij}:=\lambda_i\cap \lambda_j$ uses $n-3$ maximal lines, i.e., all maximal lines except  $\lambda_i$ and $\lambda_j$ (see the case of the node $A_{34},$ i.e., $(i,j)=(3,4)$ and $(k_1,k_2)=(1,2)$ in Fig. \ref{Def3}). Next we identify the remaining three used lines.
Again we have three $1_m$-nodes in each of two maximal lines, six in all. Now four of these six nodes lie in the two used $OO$-lines $\ell_{k_1}^{oo},\ell_{k_2}^{oo}$.  Finally, denote by $\ell_i^j$ the third used line passing through $D_i$ and the node $\lambda_j\cap \ell_i^{oo}.$

The set of used lines in ${\mathcal X}$ consists of the following five classes:

1) The set of $n-1$ maximal lines;

2) The set of three $OO$-lines - $\{\ell_{i}^{oo},\ 1\le i\le 3\};$

3) The set of three $DD$-lines - $\{\ell_{i}^{dd},\ 1\le i\le 3\};$

4) The set of three lines $\{\ell_{ij}, 1\le i<j\le 3\};$

5) The three sets of $n-4$ lines $\{\ell_i^j,\ 4\le j\le n-1 \},\ i=1,2,3.$

\begin{proposition}\label{1-5} The classes of lines 1)-5) are disjoint.
\end{proposition}

\begin{proof}
It is enough to show that the classes 2) -5) are disjoint.

Let us start with a line $\ell_2:=\ell_{m}^{oo}$ from the class 2).

This line cannot coincide with a line $\ell_{i}^{dd}$ from the class 3). Indeed, the latter line passes through two $D$-nodes, which, according to Proposition \ref{prp:n-1max}, (i), do not belong  to any line from the class 2).

Then $\ell_2$ cannot coincide with a line $\ell_{ij}$ from the class 4), since the latter line passes through two $1_m$-nodes which belong to two different $OO$-lines.

Next $\ell_2$ cannot coincide with a line $\ell_i^j$ from the class 5). Indeed, the latter line passes through a $D$-node, which, as was mentioned above, does not belong  to any line from the class 2).

Now consider a line $\ell_3:=\ell_{m}^{dd}$ from the class 3). This line cannot coincide with a line $\ell_{ij}$ from the class 4), since the latter line certainly does not pass through two $D$-nodes: $D_i$ and $D_j.$

Next let us show that $\ell_3$ cannot coincide with a line $\ell_i^j$ from the class 5). Set $\{m,k_1,k_2\}=\{1,2,3\}.$ We have that the line $\ell_{m}^{dd}$ intersects the lines $\ell_{k_1}^{oo}$ and $\ell_{k_2}^{oo}$ at their intersection node $O_m.$ On the other hand the line $\ell_i^j$ intersects $\ell_{i}^{oo}$ at an $1_m$-node. Thus, if the given two lines coincide, then we readily conclude that $i\neq k_1, k_2,$ hence $i=m.$ Now the line $\ell_m^j$ passes through the node $D_{m},$ while $\ell_3$ does not pass through $D_{m}.$ Hence they cannot coincide.

Finally, let us show that a line $\ell_{ij}$ from the class 4) cannot coincide with a line $\ell_{i'}^{j'}$ from the class 5). Set $\{i,j,k\}=\{1,2,3\}.$ We have that the line $\ell_{ij}$ intersects the maximal lines $\lambda_i$ and $\lambda_j$ at nodes belonging to $OO$-lines and the line $\ell_{i'}^{j'}$ passes through the node $D_{i'}.$ Thus, if the given two lines coincide, then we conclude that $i'\neq i,j,$ hence $i'=k.$
Thus it suffices to show that line $\ell_{ij}$ does not pass through the node $D_k:$
\begin{lemma} \label{pappus} Let ${\mathcal X}$ be a $GC_n$ set of defect $3.$ Then the line $\ell_{ij},\ 1\le i<j\le 3,$ does not pass through the node $D_k,$ where $(i,j,k)=(1,2,3).$
\end{lemma}
\begin{figure}[ht] 
\centering
\includegraphics[width=9cm,height=3.5cm]{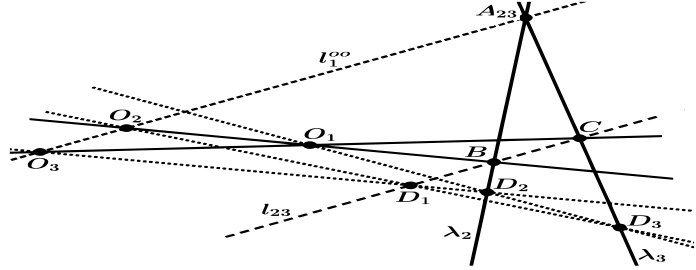}
\caption{May the node $D_1$ belong to the line $\ell_{23}?$}\label{pic15}
\end{figure}
Suppose, by way of contradiction, that the line $\ell_{ij}=\ell_{23},$ say, passes through the node $D_1$ (see Fig. \ref{pic15}).
Let us apply the  Pappus  theorem for the pair of triple collinear nodes here:
$\{D_1, B, C\};$ and   $\{O_1, D_3, D_2\}.$
Denote by $\ell(P,Q)$ the line passing through the points $P$ and $Q.$ Observe that $\ell(D_1,D_2) \cap \ell(C,O_1)=O_3,\ \ell(D_1,D_3) \cap \ell(B,O_1)=O_2,\ \ell(C,D_3) \cap \ell(B,D_2)=A_{23}.$
Thus, according to the Pappus theorem we get that the triple of nodes $\{O_2,O_3,A_{23}\}$ is collinear, i.e., the $OO$-line $\ell_1^{oo}$ passes through the point of intersection of maximal lines $\lambda_1$ and $\lambda_2,$  thus contradicting Proposition \ref{prp:n-1max}, (ii) (cf. the last part of Proposition 3.8's proof in \cite{HV2}).
\end{proof}

Then, let us verify that within each class the differently denoted lines are different.
This is evident for the classes 1), 2), 3) and 5). Thus assume that two lines in class 4) coincide: $\ell_{ij}=\ell_{i'j'}.$ We have that each of these lines passes through two
$1_m$-nodes belonging to different $OO$-lines. Hence there is an $O$-line which intersects both lines at nodes of ${\mathcal X}.$ Hence, without loss of generality, we may assume that $i=i'.$ Now suppose by way of contradiction that $j\neq j'.$ Then we have that two $OO$-lines do not intersect the maximal line $\lambda_i$ at nodes of ${\mathcal X}$, which contradicts   Proposition \ref {prp:n-1max}, (i).

Next we show that the lines in the class 2) are proper.
\begin{proposition} \label{prp1:def3}
Suppose that ${\mathcal X}$ is a $GC_n$ set of defect $3$ and $\ell=\ell_{i}^{oo},$ where $i\in\{1,2,3\}.$
Suppose also that ${\mathcal X}\downarrow_\ell \hat{\mathcal X}$. Then $\hat{\mathcal X}$ is a $GC_{n-1}$ set of defect $2$, where $\ell$ is a maximal line. The set ${\mathcal X}^\ell=\hat{\mathcal X}\setminus \ell$ is a $GC_{n-2}$ set of defect $2$. \end{proposition}
\begin{proof} We have that
$\hat{\mathcal X}={\mathcal X}\setminus \lambda_i.$
It is easily seen that $\ell$ is a newly emerged maximal line of $\hat{\mathcal X}.$ In view of Proposition \ref{prp:CG-1} we conclude that
$\hbox{def}(\hat{\mathcal X})=2.$ It remains to note that the set ${\mathcal X}^\ell$ is a $GC_{n-2}$ set satisfying the conditions of Proposition \ref{prp:nmax}, where the $O$-lines are $\ell_{j}^{oo},\ \ell_{k}^{oo},\ \ell_{i}^{dd},$  and $\{i,j,k\}=\{1,2,3\}.$
\end{proof}

Next we  show that the lines in the classes 3), 4), and 5), are proper $(-1), (-2),$ and $(-2),$ and are used by $3,1,$ and $1$ node, respectively.
We also identify the $\hat 2_m$-nodes in these lines.

\begin{proposition} \label{prp345}
Suppose that ${\mathcal X}$ is a $GC_n$ set of defect $3.$  Set \\ $\ell_1:=\ell_{k}^{dd},\  \ell_2:=\ell_{ij},$ and $\ell_3:=l_k^m,$  where $1\le k\le3,\ 4\le m\le n-1,$  and $\{i,j,k\}=\{1,2,3\}.$ Set also $ \{k,k_3,k_4\}=\{1,2,3\} .$
Then the following points are the only candidates for $\hat 2_m$- nodes in the respective lines $\ell_r, r=1,2,3$: \\ $S_1:=\ell_{k}^{oo}\cap\ell_1,\ S_2:=\ell_{k}^{oo}\cap\ell_2,\ S_3:=\ell_{k_3}^{oo}\cap\ell_3,\ S_4:=\ell_{k_4}^{oo}\cap\ell_3.$ \\ Also, for each  $S=S_q$ and the respective set $\hat{\mathcal X}=\hat{\mathcal X}(\ell_r)$ we have that
\begin{equation*}\label{s}S\in \hat{\mathcal X}\Leftrightarrow S\ \hbox{is an $1_m$-node in}\ {\mathcal X}\Leftrightarrow  S\ \hbox{is a $\hat 2_m$-node}\ \Leftrightarrow S=\ell^{oo}\cap\lambda^*,\end{equation*}
where $\ell^{oo}\in\Pr({\mathcal X})\cap M(\hat{\mathcal X})$ and $\lambda^*\in M({\mathcal X}).$\\
Next, we have that $\hbox{def}[\hat{\mathcal X}(\ell_1)]=\hbox{def}[\hat{\mathcal X}(\ell_2)]=2,\ \hbox{def}[\hat{\mathcal X}(\ell_3)]=1,$ and\\ $\#{\mathcal X}^{\ell_1}=3,\ \#{\mathcal X}^{\ell_2}=\#{\mathcal X}^{\ell_3}=1.$\\ Further, the lines $\ell_1,\ell_2,$ and $\ell_3,$ are proper $(-1),(-2),$ and $(-2),$ respectively:
 \begin{equation}\label{ell123}{\mathcal X}^{\ell_1} = \hat{\mathcal X}\setminus ({\mathcal C}_1\cup \ell_1), \
{\mathcal X}^{\ell_2} = \hat{\mathcal X}\setminus ({\mathcal C}_2\cup\ell_1\cup \ell_2),\ {\mathcal X}^{\ell_3} = \hat{\mathcal X}\setminus ({\mathcal C}_3\cup{\mathcal C}_4 \cup \ell_3),\end{equation}
where ${\mathcal C}_i$ is the d/a item.\\
Moreover, the following assertions hold for the lines $\ell_1$ and $\ell_2:$\\
If $S_i\notin\hat{\mathcal X}(\ell_i),$  then $\hat{\mathcal X}(\ell_i)$ is a $GC_3$ set,  and ${\mathcal C}_i=\ell_{k}^{oo},\ i=1,2.$\\
If  $S_i\in\hat{\mathcal X}(\ell_i),$ then $\hat{\mathcal X}(\ell_i)$ is a $GC_4$ set,  ${\mathcal C}_i=\ell_{k}^{oo}\cup \lambda_i^*,$ where $\lambda^*_i\in M({\mathcal X})$ and $S_i\in\lambda^*_i,\ i=1,2.$\\
The following assertions hold for the line $\ell_3:$\\
If $S_3,S_4\notin\ell_3$ then $\hat{\mathcal X}(\ell_3)$ is $GC_3$ set,  and
${\mathcal C}_i=\ell_{k_i}^{oo},\ i=3,4.$\\
If $S_3\in\ell,\ S_4\notin\ell$ then $\hat{\mathcal X}(\ell_3)$ is $GC_4$ set, ${\mathcal C}_3=\ell_{k_3}^{oo}\cup \lambda^*,\ {\mathcal C}_4=\ell_{k_4}^{oo}.$\\
If  $S_3,S_4\in\ell_3$ then $\hat{\mathcal X}(\ell_3)$ is $GC_5$ set, ${\mathcal C}_i=\ell_{k_i}^{oo}\cup \lambda_i^*,$  where  $\lambda_i^*\in M({\mathcal X})$ and $S_i\in \lambda_i^*, \ i=3,4.$
\end{proposition}


\begin{proof} Note that, in view of Proposition \ref{prp:n-1max}, (i), all $1_m$-nodes of ${\mathcal X},$ except the three $D$-nodes, belong to the three $OO$-lines. Note also that no $2_m$-node of a line $\ell$ belongs to $\hat {\mathcal X}(\ell).$

The line $\ell_1$ passes through two $D$-nodes and the point of intersection of two $OO$-lines:
 the node $O_k.$ Thus $S_1=\ell_1\cap\ell_{k}^{oo}$ is the only candidate for a third $1_m$-node in $\ell_1.$

In view of Lemma \ref{pappus} the line $\ell_2$ does not pass through $D$-nodes. It passes through two $1_m$-nodes which are intersection points with $OO$-lines. Thus $ S_2=\ell_2\cap\ell_{k}^{oo}$ is the only candidate for a third $1_m$-node in $\ell_2.$

Suppose first that the lines $\ell_1$ and $\ell_2$ do not have a third $1_m$-nodes
(see the case of $\ell_1=\ell_3^{dd}$ and $\ell_2=\ell_{12},$ i.e., $(i,j,k)=(1,2,3)$ in Fig. \ref{Def3}).
The following holds for both lines $\ell=\ell_1$ and $\ell=\ell_2.$

Observe that $\ell$ has exactly two $1_m$-nodes coinciding with the intersection points with the maximal lines $\lambda_i$ and $\lambda_j.$ Thus, in view of Proposition \ref{abab}, we have that $\hat{\mathcal X}$ is a $GC_{3}$ set. It consists of the node $A_{ij},$ six $1_m$-nodes: three and three lying in the lines $\lambda_i$ and $\lambda_j,$ and the three $0_m$-nodes: $O_1, O_2,O_3.$
Here the line $\ell_{k}^{oo}$ with $4$ nodes, including two $O$ nodes, is an $\ell$-disjoint maximal.

Next, the line $\ell_1$ is a maximal line in the set  $\hat{\mathcal X}\setminus \ell_{k}^{oo}$ and
the line $\ell_2$ is a maximal line in the set  $\hat{\mathcal X}\setminus \left(\ell_{k}^{oo}\cup\ell_1\right).$
Hence we readily get that the first two relations of \eqref{ell123} hold with ${\mathcal C}_1={\mathcal C}_2=\ell_{k}^{oo}.$

Now note that $M(\hat{\mathcal X})=\{\lambda_i,\lambda_j, \ell_{k}^{oo}\}$ and hence $\hbox{def}({\mathcal X})=2.$ Indeed, if there is another maximal line then it is easily seen that it has to pass through two $O$ nodes and $2$ other nodes in $\lambda_i$ and $\lambda_j.$ But clearly the two $OO$-lines: $\ell_{i}^{oo}$ and $\ell_{j}^{oo},$ pass through only $3$ nodes of $\hat{\mathcal X}.$ Note also that $\ell_1$ is a $3$-node line in $\hat{\mathcal X}$ and $\ell_2$ is a $2$-node line in $\hat{\mathcal X},$ without any $\hat 2_m$-node.

Next, suppose that each of the lines $\ell_q, q=1,2,$ has a third $1_m$-node. Then this node is $S_q=\lambda^*_q \cap \ell_{k}^{oo},$ where $\lambda^*_q\in M({\mathcal X}).$
Note that $S_q$ is the only $\hat 2_m$-node in $\ell_q.$ Now the number of $1_m$-nodes in each of the lines equals to three. Thus, in view of Proposition \ref{abab}, we have that $\hat{\mathcal X}(\ell_q)$ is a  $GC_{4}$ set.
Here $\ell_{k}^{oo}$ and $\lambda^*_q$ form a pair of $\ell$-adjoint maximal lines and
we readily get  that the first two relations of \eqref{ell123} hold with ${\mathcal C}_q=\ell_{k}^{oo}\cup\lambda^*_q ,\ q=1,2.$

Note that $M(\hat{\mathcal X}(\ell_q))=\{\lambda_i,\lambda_j, \ell_{k}^{oo},\lambda_q^*\}$ and hence $\hbox{def}[\hat{\mathcal X}(\ell_q)]=2,\ q=1,2.$ Indeed, if there is more than four maximal lines then it is easily seen that $\hat{\mathcal X}\setminus \lambda_q^*$ would have more than three maximal lines, which contradicts the consideration in the previous case.

Now consider the line $\ell_3.$
It passes through two $1_m$-nodes, namely $D_k$ and the point of intersection of $\lambda_m$ with one $OO$-line  $\ell_{k}^{oo}.$ Thus the points of intersections with the other two $OO$-lines, i.e., $S_3$ and $S_4,$ are the only candidates for a third and forth $1_m$-nodes in $\ell_3.$

Suppose first that the line $\ell_3$ does not have a third and forth $1_m$-nodes (see the case of $\ell_3=\ell_3^4,$ i.e., $(k,m)=(3,4), (k_3,k_4)=(1,2)$ in Fig. \ref{Def3}).
Then $\ell_3$ has exactly two $1_m$-nodes.  Thus, in view of Proposition \ref{abab}, we have that $\hat{\mathcal X}$ is a $GC_{3}$ set. It consists of the node $A_{ij},$ six
$1_m$-nodes: three and three lying in the lines $\lambda_i$ and $\lambda_j,$ and the three $0_m$-nodes: $O_1, O_2,O_3.$
Here the lines $\ell_{k_q}^{oo},\ q=3,4,$ with $4$ nodes, including two $O$ nodes, are $\ell$-disjoint maximal lines. Next, the line $\ell_3$ is a maximal line in the set  $\hat{\mathcal X}\setminus (\ell_{k_3}^{oo}\cup\ell_{k_4}^{oo}).$
Hence we readily get that the third relation of \eqref{ell123} holds with ${\mathcal C}_q=\ell_{k_q}^{oo},\ q=3,4.$

Now we have that $M(\hat{\mathcal X})=\{\lambda_i,\lambda_j, \ell_{k_3}^{oo}, \ell_{k_4}^{oo}\}$ and hence $\hbox{def}({\mathcal X})=1.$ Note that $\ell_3$ is a $2$-node line in $\hat{\mathcal X},$ without any $\hat 2_m$-node.

Next, suppose that the line $\ell_3$ has a third $1_m$-node, say, $S_3.$ Then we have that $S_3=\lambda^*_3 \cap \ell_{k_3}^{oo},$ where $\lambda^*_3\in M({\mathcal X}).$
Now the number of $1_m$-nodes in $\ell_3$ equals to three. Thus, in view of Proposition \ref{abab}, we have that $\hat{\mathcal X}(\ell_q)$ is a  $GC_{4}$ set.
Here $\ell_{k_3}^{oo}$ and $\lambda^*_3$ form a pair of $\ell_3$-adjoint maximal lines and
we readily get  that the third relation of \eqref{ell123} holds with ${\mathcal C}_3=\ell_{k_3}^{oo}\cup\lambda^*_3$ and ${\mathcal C}_4=\ell_{k_4}^{oo}.$ Note that $S_3$ is the only $\hat 2_m$-node in $\ell_3.$

Then we have that $M(\hat{\mathcal X})=\{\lambda_i,\lambda_j, \ell_{k_3}^{oo}, \ell_{k_4}^{oo},\lambda^*_3\}$ and hence $\hbox{def}({\mathcal X})=1.$ Indeed, if there is more than five maximal lines then it is easily seen that $\hat{\mathcal X}\setminus \lambda^*_3$ would have more than four maximal lines, which contradicts the consideration in the previous case.

Finally, suppose that the line $\ell_3$ has four $1_m$-nodes. Then we have that $S_q=\lambda^*_q \cap \ell_{k_q}^{oo},$ where $\lambda^*_q\in M({\mathcal X}),\ q=3,4.$
Now the number of $1_m$-nodes in $\ell_3$ equals to four. Thus, in view of Proposition \ref{abab}, we have that $\hat{\mathcal X}(\ell_q)$ is a  $GC_{5}$ set.
Here $\ell_{k_q}^{oo}$ and $\lambda^*_q,\ q=3,4,$ form two pairs of $\ell$-adjoint maximal lines and
we readily get  that the third relation of \eqref{ell123} holds with ${\mathcal C}_q=\ell_{k_q}^{oo}\cup\lambda^*_q,\ q=3,4.$ Note that in this case $S_3$ and $S_4$ are the only $\hat 2_m$-nodes in $\ell_3.$
Note also that $M(\hat{\mathcal X})=\{\lambda_i,\lambda_j, \ell_{k_3}^{oo}, \ell_{k_4}^{oo},\lambda^*_3, \lambda^*_4\}$ and hence $\hbox{def}({\mathcal X})=1.$
\end{proof}
Denote by $n_i$ the number of line usages in the class $i),\ i=1,\ldots,5.$ Then we have that the total number of line-usages equals:

$n_1+\cdots+n_5= (n-1)\binom{n+1}{2}+ 3 \binom{n}{2}+9 + 3 + 3(n-4)= n\binom{n+2}{2}.$

\section{Generalized principal lattices: defect $n-1$ sets}

A principal lattice is defined as an affine image of the set (see Fig. \ref{pic123b})
$$PL_n:=\left\{(i,j)\in{\mathbb N}_0^2 : \quad i+j\le n\right\}.$$
Let us set $I=\{0,1,\ldots,n\}.$ Observe that  the following $3$ sets of $n+1$ lines, namely $\{x=i:\ i\in I\},\  \{y=j:\ j\in I\},$ and $\{x+y=n-k:\ k\in I\},$  intersect at $PL_n.$
We have that $PL_n$ is a $GC_n$ set. Moreover, the following is the fundamental polynomial of the node $(i_0,j_0)\in PL_n:$
\begin{equation}\label{aaabbc} p_{i_0j_0}^\star(x,y)  =\prod_{0\le i<i_0,\ 0\le j<j_0,\ 0\le k<k_0} (x-i)(y-j)(x+y-n+k),\end{equation}
where $k_0=n-i_0-j_0.$
\begin{figure}[ht] 
\centering
\includegraphics[width=8cm,height=2.cm]{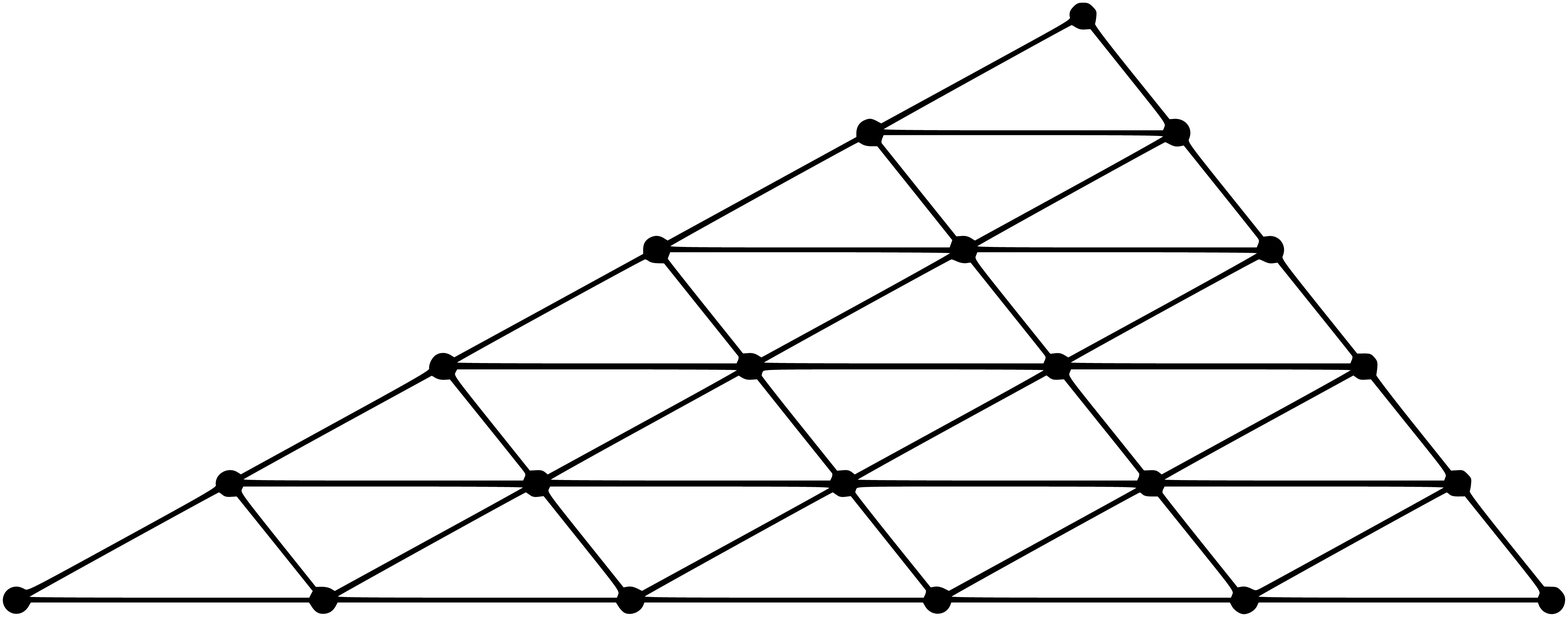}
\caption{A principal lattice $PL_5.$}\label{pic123b}
\end{figure}
Next let us bring the definition of the generalized principal lattice due to Carnicer, Gasca and God\'es (see \cite{CG05}, \cite{CGo06}):
\begin{definition}[\cite{CGo06}] \label{def:GPL} A node set ${\mathcal X}$ is called a generalized principal lattice, briefly $GPL_n,$ if there are $3$ sets of lines each containing $n+1$ lines
\begin{equation}\label{aaagpl}\ell_i^j:=\ell_i^j({\mathcal X}),\ i=0,1,\ldots,n, \ j=0,1,2,\end{equation}
such that the $3n+3$ lines are distinct,
$$\ell_i^0\cap \ell_j^1\cap \ell_k^2\cap {\mathcal X} \neq \emptyset \iff i+j+k=n,$$
and ${\mathcal X}=\left\{x_{ijk}\ |\ x_{ijk}:=\ell_i^0\cap \ell_j^1\cap \ell_k^2,
0\le i,j,k\le n, i+j+k=n\right\}.$
\end{definition}

Observe that if $0\le i,j,k\le n,\ i+j+k=n$ then the three lines $\ell_i^0, \ell_j^1,  \ell_k^2,$ intersect at a node $x_{ijk}\in {\mathcal X}.$ This implies that each node of ${\mathcal X}$ belongs to only one line of each of the three sets of $n+1$ lines. Therefore $\#{\mathcal X}=(n+1)(n+2)/2.$

One can find readily, as in the case of $PL_n$, the fundamental polynomial of each node $A=x_{i_0j_0k_0}\in {\mathcal X},\ i_0+j_0+k_0=n:$
\begin{equation}\label{aaabbc} p_A^\star  =\prod_{0\le i<i_0,\ 0\le j<j_0,\ 0\le k<k_0}\ell_{i}^0 \ell_{j}^1 \ell_{k}^2.\end{equation}
Thus ${\mathcal X}$ is a $GC_n$ set.

Let us bring a characterization for $GPL_n$ set due to Carnicer and God\'es:

\begin{theorem}[\cite{CGo06}, Thm. 3.6]\label{th:CGo06} Assume that GM Conjecture holds for all degrees up to $n-3$. Then the following statements are equivalent:
\begin{enumerate}
\item
${\mathcal X}$ is generalized principal lattice of degree n;
\item
${\mathcal X}$ is a $GC_n$ set with $\#M(\mathcal X)=3.$
\end{enumerate}
\end{theorem}

From \eqref{aaabbc} we have that the only used lines in the generalized principal lattice ${\mathcal X}$ are

1) The three sets of $n$ lines $\ell_{s}^r({\mathcal X}),$ where  $0\le s\le n-1,\ r=0,1,2.$

Let $\ell:= \ell_{n-k+1}^r({\mathcal X})$ be a $k$-node line, $k\ge 2.$
Then it is easily seen that
\begin{equation}\label{newton}{\mathcal X}^\ell= {\mathcal X}\setminus (\ell_0^r\cup \ell_1^r\cup\cdots\cup \ell_{n-k+1}^r).
\end{equation} 
Hence 
\begin{equation}\label{newton2}{\mathcal X}^\ell\ \hbox{is a}\ GC_{k-2}\ \hbox{subset of}\ {\mathcal X}\ \hbox{and}\ \#{\mathcal X}^\ell=\binom{k}{2}.\end{equation}

Next suppose that $\ell=\ell_{i}^r$ is any used line, which is not a maximal line, i.e., $1\le k\le n-1.$ 
It is easily seen that the maximal line $\lambda_0:=\ell_{0}^r$ does not intersect $\ell$ while the remaning two maximal lines intersect $\ell$ at two different nodes.
Thus for $\ell$ there is only one reduction, which is disjoint one. Therefore we get
\begin{equation}\label{newton3}\hat {\mathcal X}(\ell)={\mathcal X}\setminus \lambda_0.
\end{equation}
From here we readily get that the $3$ $n$-node lines $\ell_{1}^r, r=0,1,2,$ are the proper lines of  the generalized principal lattice ${\mathcal X}.$ On the other hand, according to Corollary 4.4 \cite{HV}, there are exactly three $n$-node lines, provided that $n\ge 4.$ Therefore we have that
\begin{equation}\label{newtonpr} \hbox{Pr}(\mathcal X)=\{\ell_{1}^0,\ell_{1}^1,\ell_{1}^2\}=N(\mathcal X),\end{equation}
where the last equality holds if $n\ge 4.$

Now, in view of Lemma \ref{pluslem}, we conclude that
\begin{equation}\label{newton4}\hbox{defect} \hat {\mathcal X}(\ell) = \hbox{defect}{\mathcal X} -1,
\end{equation}
for any used line $\ell\notin M({\mathcal X})$ in defect $n-1$ set ${\mathcal X}.$

Note that for the maximal line $\lambda_0,$ in view of Lemma \ref{lem:CG1}, we get that 
$\#(\lambda_0\cap {\mathcal X}^{\ell}) = 0$ and for each remaining maximal line $\lambda,$ in view of the relation \ref{newton}, we get that
$\#(\lambda\cap {\mathcal X}^{\ell}) = k-1.$

Next we have that the total number of line-usages here equals: 

$3\left[\binom{n+1}{2}+\binom{n}{2}+\cdots+\binom{2}{2}\right]=3\binom{n+2}{3}= n\binom{n+2}{2}.$

\section{On the properties of the set ${\mathcal X}^\ell$ \label{s:conj1}}

For $GC_n$ sets of defect $n-1$ there is a simple formula for the
set ${\mathcal X}^\ell,$ namely \eqref{newton}, which yields its basic properties given in \eqref{newton2}. For other $GC_n$ sets we have the following
\begin{theorem}\label{th:xl}
 Let ${\mathcal X}$ be a $GC_n$ set, 
${\ell}\notin M({\mathcal X})$ be a used line and ${\mathcal X}\downarrow_\ell\hat{\mathcal X}.$  Assume that $GM$ Conjecture holds for all degrees up to $n$.
Then we have that
\begin{equation}\label{def-1} \hbox{def}(\hat{\mathcal X})=\hbox{def}({\mathcal X})-1 \quad \hbox{or}\quad  \hbox{def}({\mathcal X})-2.
\end{equation}
 Next, by assuming that $\hbox{def}({\mathcal X})\neq n-1$ if $n\ge 5,$ and the line $\ell$ is not proper, we get that $\ell$ is proper $(-r)$ and has at most $r \ \hat 2_m$-nodes, where $r=1$ or $2.$ Also we have that the set ${\mathcal X}^\ell$ is $GC_0$ or $GC_1$ set and the set $\hat{\mathcal X}$ is a $GC_k$ set with $k\le 5.$\\
Moreover, if $S\in \ell$ is a $\hat 2_m$-node,  then $S=\ell^o\cap\lambda,\ \hbox{where}\ \ell^o\in Pr({\mathcal X})\cap M(\hat{\mathcal X})\ \hbox{and}\ \lambda\in M({\mathcal X}).$
\end{theorem}
\begin{proof} The proof is divided into five cases.

{\bf Case 1.} $\hbox{def}({\mathcal X})=0.$

In this case, all used lines are maximal.

{\bf Case 2.} $\hbox{def}({\mathcal X})=1.$

In this case we have two classes of used lines and all lines in the class 1) are maximal. Then, by Proposition \ref{prp:def1}, all lines in the class 2) are proper
and $\hbox{def}(\hat{\mathcal X})=0.$

{\bf Case 3.} $\hbox{def}({\mathcal X})=2.$

In this case we have three classes of used lines and all lines in the class 1) are maximal.
Then, by Proposition \ref{prp1:def2}, all the lines in the class 2) are proper
and $\hbox{def}(\hat{\mathcal X})=1.$

Thus it remains to consider a line
$\ell:=\ell_{ij}$ belonging to the class 3).
According to Proposition \ref{prp2:def2} the line $\ell$ is proper $(-1), \hbox{def}(\hat{\mathcal X})=1$ and $\#{\mathcal X}^\ell=1.$  Also, we have that $\hat{\mathcal X}$ is a $GC_2$ or $GC_3$ set if  $\ell$ contains no $\hat 2_m$-node or contains a $\hat 2_m$-node, respectively.

{\bf Case 4.} $\hbox{def}({\mathcal X})=3.$

In this case we have five classes of used lines and all lines in the class 1) are maximal.
Then, by Proposition \ref{prp1:def3}, all the lines in the class 2) are proper
and $\hbox{def}(\hat{\mathcal X})=2.$

Now consider lines $\ell_1:=\ell_k^{dd}$ and $\ell_2:=\ell_{ij}$ belonging to the classes 3) and 4), respectively, where $\{i,j,k\}=\{1,2,3\}.$
According to Proposition \ref{prp345} the line $\ell_i$ is proper $(-i)$ and $\hbox{def}[\hat{\mathcal X}(\ell_i)]=2,\ i=1,2.$
We also have that ${\mathcal X}^{\ell_1}$ and ${\mathcal X}^{\ell_2}$ are $GC_1$ and $GC_0$ sets, respectively.

Then $\hat{\mathcal X}$ is a $GC_3$ or $GC_4$ set if $\ell_i$ contains no $\hat 2_m$-node or
contains a $\hat 2_m$-node, respectively $i=1,2.$

Next consider a line $\ell:=\ell_i^j$ belonging to the class 5), where $1\le i\le 3,\ 4\le j\le n-1.$ According to Proposition \ref{prp345} the line $\ell$ is proper $(-2),$ $\hbox{def}[\hat{\mathcal X}(\ell)]=1$ and $\#{\mathcal X}^\ell = 1.$

Then $\hat{\mathcal X}$ is a $GC_3,$ $GC_4,$ and a $GC_5$ set if $\ell$ contains zero, one, two $\hat 2_m$-nodes, respectively.

{\bf Case 5.} $\hbox{def}({\mathcal X})=n-1.$ This case follows  simply from the relation \eqref{newton4}.

Finally note that the ``Moreover" statement follows from Propositions \ref{prp2:def2} and \ref{prp345}.
\end{proof}
Note that the equality
\begin{equation*}\label{def-2} \hbox{def}(\hat{\mathcal X})= \hbox{def}({\mathcal X})-2
\end{equation*}
holds only for the above lines $\ell_i^j$  belonging to the class 5) of the used lines of  defect $3$ sets.

Similarly to the proof of Theorem \ref{th:xl}, by using Propositions \ref{prp2:def2} and \ref{prp345}, one can verify the following
\begin{remark}  \label{rem} Assume that the hypotheses of Theorem \ref{th:xl} hold, $\hbox{def}({\mathcal X})\neq n-1$ if $n\ge 5,$ and in addition  $\ell$ is a proper $(-r)$ line, where $r=1$ or $2.$ Then all $(-1), (-2)$ disjoint and adjoint items are proper lines and pairs of proper lines with maximal lines of ${\mathcal X},$ respectively. The only exception is for the class 4) of the used lines of defect $3$ sets, where the $(-2)$ item is necessarily disjoint and proper $(-1)$ line.
\end{remark}

Next let us mention a result from  \cite{HV2} (Corollary 4.1), which states that if $n\ge 6$ then for any $k$-node line in a $GC_n$ set ${\mathcal X},\ 2\le k\le n,$
there is either $\ell$-disjoint maximal line or a pair of $\ell$-adjoint maximal lines.
Below we strengthen this result in the case of used lines.
\begin{corollary} Let ${\mathcal X}$ be a $GC_n$ set,
${\ell}$ be a $k$-node used line, $2\le k\le n.$  Assume that $GM$ Conjecture holds for all degrees up to $n$.
Then there is either $\ell$-disjoint maximal line or a pair of $\ell$-adjoint maximal lines.
\end{corollary}
\begin{proof}
It suffices to verify that for any used line $\ell$ we have that ${\mathcal X}\neq \hat{\mathcal X}(\ell).$ This follows from the statement $\hbox{def}(\hat{\mathcal X})\le \hbox{def}({\mathcal X})-1$ of Theorem \ref{th:xl}. 
\end{proof}

\begin{theorem}\label{th:2}
 Let ${\mathcal X}$ be a $GC_n$ set and ${\ell}$ be a $k$-node used line. Assume that $\ell$ contains exactly $r$ $2_m$-nodes and $\hat r$ $\hat 2_m$-nodes. Assume that $GM$ Conjecture holds for all degrees up to $n$.  Then ${\mathcal X}^\ell$ is a $GC_{s-2}$ set and hence $\#{\mathcal X}^\ell=\binom{s}{2},$ where $s=k-r-\hat r.$
Moreover, for any used line $\ell$ we have that $\hat r\le 2.$  Furthermore, $\hat r=0$ if $\#{\mathcal X}^\ell>3.$
\end{theorem}

\begin{proof} First note that Theorem \ref{th:xl} implies the statements in the parts ``Moreover" and ``Furthermore", if $\hbox{defect}{\mathcal X}\neq n-1.$
Then, in view of relation \eqref{max} and Proposition \ref{properties}, we get that Theorem holds with  $r=\hat r=0,$ if the line $\ell$ is a maximal line.

Next, in view of Lemmas \ref{lem:CG1} and \ref{lem:CG2}, we get that Theorem holds with  $\hat r=0,$ if the line $\ell$ is a proper line.
Thus Theorem holds if ${\mathcal X}$ is a Chung-Yao or Carnicer-Gasca lattice, since then, in view of Proposition \ref{prp:def1}, all used lines are maximal or proper lines.

If ${\mathcal X}$ is a defect $2$ set then, in view of Proposition \ref{prp1:def2}, all lines in the classes 1) and 2) are maximal and proper, respectively.  Thus consider a line
$\ell:=\ell_{ij}$ belonging to the class 3).
According to Proposition \ref{prp2:def2} the line $\ell,$ not counting a $\hat 2_m$-node, is $2$-node line in $\hat{\mathcal X}$ and $\#{\mathcal X}^\ell=1.$ Also we have that  $s=k-r-\hat r=2.$

If ${\mathcal X}$ is a defect $3$ set then, in view of Proposition \ref{prp1:def3}, all lines in the classes 1) and 2) are maximal and proper, respectively.  Thus consider lines $\ell_1:=\ell_k^{dd}$ and $\ell_2:=\ell_{ij}$ belonging to the classes 3) and 4), respectively, where $\{i,j,k\}=\{1,2,3\}.$
According to Proposition \ref{prp345} the line $\ell_1,$ not counting a possible $\hat 2_m$-node, is $3$-node line in $\hat{\mathcal X}$ and $\#{\mathcal X}^{\ell_1}=3.$ Also we have that  $s=k-r-\hat r=3.$ While the line $\ell_2,$ not counting a possible $\hat 2_m$-node, is $2$-node line in $\hat{\mathcal X}$ and $\#{\mathcal X}^{\ell_2}=1.$ Also we have that  $s=k-r-\hat r=2.$

Next, consider a line $\ell:=\ell_i^j$ belonging to the class 5), where $1\le i\le 3,\ 4\le j\le n-1.$
According to Proposition \ref{prp345} the line $\ell,$ not counting possible two $\hat 2_m$-nodes, is a $2$-node line in $\hat{\mathcal X}$ and $\#{\mathcal X}^\ell=1.$ Also we have that  $s=k-r-\hat r=2.$

Finally, let  ${\mathcal X}$ be a generalized principal lattice, i.e., $\hbox{def}{\mathcal X}=n-1.$ Consider a line $\ell,$ which is not maximal or proper. Then we have that $r=\hat r=0.$ On other hand, in view of the formula \eqref{newton2}, we have that $\#{\mathcal X}^\ell=\binom{k}{2}.$
\end{proof}

From Theorem \ref{th:2} we readily get

\begin{corollary}Let ${\mathcal X}$ be a $GC_n$ set and ${\ell}$ be a $k$-node used line. Assume that $\ell$ contains exactly $r$ $2_m$-nodes.  Assume that $GM$ Conjecture holds for all degrees up to $n$. Then $\#{\mathcal X}^\ell\in\left\{1,3,\binom{k-r}{2}\right\}.$
\end{corollary}

Indeed, if $\hat r=0$ then $\#{\mathcal X}^\ell=\binom{k-r}{2}.$ While if $\hat r\neq 0$ then $\hbox{def}({\mathcal X})\neq n-1,$ and, in view of Theorem  \ref{th:xl}, we get that 
$\#{\mathcal X}^\ell\le 3,$ i.e., $\#{\mathcal X}^\ell=1$ or $3.$

Below, we restate a result from \cite{HV2} (Theorem 3.1). In the ``Moreover" and ``Furthermore" parts we complement it.\\  Let us call the maximal line of ${\mathcal X}$ passing through an $\hat 2_m$-node in $\ell:$ \emph{$\ell$-special maximal.}
\begin{theorem}\label{maxpr}
 Let ${\mathcal X}$ be a $GC_n$ set,
${\ell}$ be a line with $\#{\mathcal X}^\ell=\binom{s}{2},\ s\ge 2.$  Assume that $GM$ Conjecture holds for all degrees up to $n$. Then for any maximal line $\lambda$ we have that
 $\#(\lambda\cap{\mathcal X}^\ell)=s-1,$ or $0.$\\ Moreover, the latter case:
\begin{equation}\label{aaa}\#(\lambda\cap{\mathcal X}^\ell)=0,\end{equation} holds, if and only if
\begin{enumerate}
\item
$\lambda$ is an $\ell$-disjoint maximal line, or
\item
$\lambda$ is one of $\ell$-adjoint maximal lines, or
\item
$\lambda$ is an $\ell$-special maximal line.
\end{enumerate}
Furthermore, for any line $\ell,$ (iii) may take place for at most two maximal lines $\lambda.$
\end{theorem}
\begin{proof}  First suppose that $\ell$ is a maximal line. Then, in view of the relation \eqref{max} and Proposition \ref{properties}, (ii), we have that $\#(\lambda\cap{\mathcal X}^\ell)=n$ for any other maximal line $\lambda.$ 

Now from Lemma \ref{lem:CG1} and \ref{lem:CG2} applied to both ${\mathcal X}$ and $\hat{\mathcal X}$ the direction ``if" of the ``Moreover'' part follows. For the direction ``only if" note first that if ${\mathcal X}$ is a defect $n-1$ set then Theorem follows from the comment in the paragraph after the relation \eqref{newton4}. Next assume that $\hbox{def}({\mathcal X})\neq n-1.$ Consider a maximal line $\lambda$ and assume that it is not $\ell$-disjoint, $\ell$-adjoint and $\ell$-special. Now, according to Definition \ref{def:reduct} and Proposition \ref{crl:minusmax}, $\lambda$ is a maximal line in $\hat{\mathcal X}.$ Then, in view of the Remark \ref{rem}, $\lambda$ is not present in any $(-1)$ or $(-2)$ d/a item. Consequently,  Proposition \ref{crl:minusmax} implies that $\lambda$ is a maximal line in the $GC_{s-2}$ set ${\mathcal X}^\ell$ and hence $\#(\lambda\cap{\mathcal X}^\ell)=s-1.$

Finally, note that the ``Furthermore'' part follows from Theorem \ref{th:xl}.
\end{proof}

\section{On proper, maximal, and $n$-node lines  \label{s:c3}}

Recall that $N({\mathcal X})$ denotes the set of all $n$-node lines of $GC_n$ set ${\mathcal X}.$ The following Proposition presents four unexpected properties of the set $N({\mathcal X}).$
\begin{proposition}[]\label{prp:N}Let ${\mathcal X}$ be a $GC_n$ set, $n\ge 4.$ Assume that $GM$ Conjecture holds for all degrees up to $n$.   Then $N({\mathcal X})\subset Pr({\mathcal X})$ and $\#N({\mathcal X}) \in\left\{0,1,2,3\right\}.$
Moreover, any $n$-node line intersects each maximal line, except possibly one, at a node of ${\mathcal X}.$
Furthermore, any two $n$-node lines intersect at a node of ${\mathcal X}.$
\end{proposition}
\begin{proof}  Assume that $\ell$ is an $n$-node line. Assume also that $\ell$ is type $(i,j,k),$  meaning that it passes through $i$ $0_m$-nodes, $j$ $1_m$-nodes, and $k$  $2_m$-nodes. Note that to prove the ``Moreover" part it suffices to verify the inequality $j+2k\ge \#M({\mathcal X})-1.$
Consider the following cases:

{\bf Case 1.} $\#M({\mathcal X})=n+2.$

In this case we have only  $2_m$-nodes and there is no $n$-node line.  Indeed, suppose by way of contradiction that $\ell$ is an $n$-node line. Then for the number of maximal lines intersecting $\ell$ we have $2n\le n+2,$ i.e., $n\le 2.$

{\bf Case 2.} $\#M({\mathcal X})=n+1.$

In this case we have only $1_m$ and $2_m$-nodes.
Thus $\ell$ is type $(0,j,k)=(0,n-k,k).$ Then we have that
$n-k+2k\le n+1.$ Hence $k=0$ or $1.$

Recall that there are $n+1$ $1_m$-nodes. Thus only one $n$-node line is possible if $n\ge 5:$ type $(0,n,0)$ or type $(0,n-1,1).$
If $n=4$ then either one type $(0,4,0)$ line is possible or two type $(0,3,1)$ lines.
In view of Proposition \ref{prp:def1} all above $n$-node lines are proper lines.


{\bf Case 3.} $\#M({\mathcal X})=n.$

In this case we have one $0_m$-node: $O.$

 First suppose that the line $\ell$ does not pass through $O,$ i.e., it is type $(0,j,k)=(0,n-k,k).$ Then we have that
$n-k+2k\le n,$ i.e., $k=0.$  In this case, in view of  Proposition \ref{prp:nmax}, (i), the number of $1_m$-nodes in $\ell$ is less than or equal to $3.$ Hence $n\le 3.$

Now suppose that the line $\ell$ passes through $O,$ i.e., is type $(1,j,k)=(1,n-1-k,k).$ Also suppose that $\ell$ is different from $O$-lines. In this case, by Proposition \ref{prp:nmax}, (i), the number of $1_m$-nodes in $\ell$ equals to $0.$ Hence $k=n-1$ and we have that $2(n-1)\le n.$ Hence $n\le 2.$

Next suppose that the line $\ell$ is an $O$-line.
For the number of $1_m$-nodes in the three $O$-lines we have $n_1+n_2+n_3=2n,$ where $n_i\ge 2$ is the number of $1_m$-nodes in $i$th $O$-line. Suppose all three lines are $n$-node lines. Observe that since $(n-1-k)+2k\le n$ then each of them has at most one $2_m$-node. Therefore we have that $3(n-2)\le 2n$ hence $n\le 6.$ Thus for $n\ge 7$ two $n$-node lines are possible in all: of types $(1,n-1,0)$ or $(1,n-2,1).$
It is easily seen that in the cases $n=4,5,6,$ all three $O$-lines can be $n$-node lines of types $(1,n-1,0)$ or  $(1,n-2,1).$
In view of Proposition \ref{prp1:def2} each $O$-line is a proper line.

{\bf Case 4.} $\#M({\mathcal X})=n-1.$

In this case we have three non-collinear $0_m$-nodes: $O_1,O_2,O_3.$

First suppose that the line $\ell$ does not pass through any $O$ node, i.e., it is type $(0,j,k)=(0,n-k,k).$ Then we get $n-k+2k\le n-1,$ i.e., $k\le -1.$

Now suppose that the line $\ell$ passes through one $O$ node, i.e., is type $(1,j,k)=(1,n-1-k,k).$ Also suppose that $\ell$ is different from $OO$-lines.
Then we have that
$n-1-k+2k\le n-1.$ Hence $k=0.$  In this case, by Proposition \ref{prp:n-1max}, (iii), the number of $1_m$-nodes in $\ell$ is less than or equal to $3$ (the case of $DD$-line with  one $\hat 2_m$-node). Thus $n-1\le 3$ and $n\le 4.$ It is easily seen that in the case of $n=4$ a $DD$-line can not have a $\hat 2_m$-node and be a $4$-node line.

Next suppose that the line $\ell$ is an $OO$-line. By Proposition \ref{prp:n-1max}, (ii), these three lines are $n$-node lines of type $(2,n-2,0).$
In view of Proposition \ref{prp1:def3} each $OO$-line is a proper line.

{\bf Case 5.} $\#M({\mathcal X})=3.$ This case follows from the relation \eqref{newtonpr}.
\end{proof}
\begin{remark} \label{remnnode} Note that all above $n$-node lines are either type $(i,j,0)$ with $i+j=n,\ j+2\cdot 0=\#M({\mathcal X})-1,$ or type $(i,j,1)$ with $i+j+1=n,\ j+2\cdot 1=\#M({\mathcal X}).$ Note that in the first case $n$-node line intersects all but one maximal line of ${\mathcal X},$ and in the second case it intersects all maximal lines of ${\mathcal X}.$
\end{remark}
Recall that in defect $1$ sets the proper lines are the lines passing through at least two of $n+1$ $1_m$-nodes. For other $GC$ sets we have the following
\begin{proposition}[]\label{pr}Let ${\mathcal X}$ be a $GC_n$ set, $\hbox{def}({\mathcal X})\neq 1,\ n\ge 4.$ Assume that $GM$ Conjecture holds for all degrees up to $n$. Then $\#Pr({\mathcal X}) \in\left\{0,3\right\}.$
\end{proposition}
\begin{proof} Indeed, there are no proper lines in Chung-Yao lattice. In the defect $2$ and defect $3$ sets the only proper lines are the three $O$-lines and three $OO$-lines, respectively.
Finally, the case of defect $n-1$ sets follows from the relation \eqref{newtonpr}.
\end{proof}

Next, we complement Proposition \ref{prp:CG-1}.   From Proposition \ref{prp:N} and Remark \ref{remnnode} we readily get the following

\begin{proposition}\label{prop:def} Let ${\mathcal X}$ be a $GC_n$ set, $n\ge 4.$ Assume that $GM$ Conjecture holds for all degrees up to $n$. Then the relation
\begin{equation}\label{defmax}\hbox{def}({\mathcal X}\setminus \lambda)=\hbox{def}({\mathcal X})-1
\end{equation}
holds for $\lambda\in M({\mathcal X})$ if and only if there is a type $(i,j,0)$ $n$-node line $\ell$ such that $\ell\cap\lambda\notin {\mathcal X}.$ Hence \eqref{defmax} holds for at most three maximal lines $\lambda.$
\end{proposition}

Note that for all other maximal lines $\lambda,$ in view of Proposition  \ref{prp:CG-1}, we have that $\hbox{def}({\mathcal X}\setminus \lambda)=\hbox{def}({\mathcal X}).$

For a defect $n-1$ set there is a simple formula for the fundamental polynomials of nodes $A\in{\mathcal X},$ namely \eqref{newton}. This formula gives the $n$ used lines of the node $A.$ For other $GC_n$ sets we have the following

\begin{proposition}\label{th:node}
 Let ${\mathcal X}$ be a $GC_n$ set and $\hbox{def}({\mathcal X})\neq n-1$ if $n\ge 5.$ 
Then we have that among $n$ lines used by a node $A\in{\mathcal X}$ at most three lines are proper, at most one line is proper $(-1),$ at most one line is proper $(-2),$  and at least $n-3$ lines are maximal.
\end{proposition}
\begin{proof}
Let us list the following possibilities of $n$ lines used by $A$:\\
\noindent 1) All $n$ lines are maximal lines; \\
2) $n-1$ lines are maximal and one line is proper;\\
3) $n-2$ lines are maximal and two lines are proper;\\
4) $n-2$ lines are maximal, one is proper and one is proper $(-1)$;\\
5) $n-3$ lines are maximal and three lines are proper;\\
6) $n-3$ lines are maximal, two are proper, and one is proper $(-2)$;\\
7) $n-3$ lines are maximal, one is proper, one is proper $(-1),$ and one is proper $(-2)$.

In view of the results in Section 3, one can readily verify the following:

If ${\mathcal X}$ is a Chung-Yao lattice then for all nodes the item 1) holds.

If ${\mathcal X}$ is a Carnicer-Gasca lattice then for all $1_m$-nodes the item 1) holds. While for all $2_m$-nodes the item 2) holds.

If ${\mathcal X}$ is defect $2$ set then for the  node $O$ the item 1) holds. For all  $1_m$-nodes the item 2) holds. For all $2_m$-nodes the item 3) or 4) holds.

If ${\mathcal X}$ is defect $3$ set then for the three $O$-nodes the item 2) holds. For the three $D$-nodes the item 3) holds. For the remaining $1_m$-nodes the item 3) or 4) holds. Finally, for all $2_m$-nodes one of the items 5), 6) or 7) holds.
\end{proof}


\begin{thebibliography}{99}

\bibitem{BH}
V. Bayramyan and H. Hakopian,
On a new property of n-correct and $GC_n$ sets, Adv. Comput. Math.,  {\bf 43}, (2017) 607--626.



\bibitem{B90}
J. R. Busch,
{A note on Lagrange interpolation in $\mathbb{R}^2$},
Rev.\ Un.\ Mat.\ Argentina, {\bf 36} (1990) 33--38.

\bibitem{CG00}
J. M. Carnicer and M. Gasca,
{Planar configurations with simple Lagrange formula},
in {\it Mathematical Methods in CAGD,}   eds. T. Lyche and L. L. Schumaker (Vanderbilt University Press, Nashville, Oslo 2000, TN, 2001) 55–-62. 

\bibitem{CG01}
J. M. Carnicer and M. Gasca,
{A conjecture on multivariate polynomial interpolation},
Rev. R. Acad. Cienc. Exactas F{\'i}s. Nat. (Esp.), Ser.~A
Mat. {\bf 95} (2001) 145--153.

\bibitem{CG03}
J. M. Carnicer and M. Gasca,
{On Chung and Yao's geometric characterization for bivariate polynomial
interpolation}, in {\it  Curve and Surface Design,} eds. T. Lyche, M.-L. Mazure, and L. L. Schumaker (Nashboro Press, Saint Malo 2002, Brentwood, 2003) 21--30.

\bibitem{CG05}
J. M. Carnicer and M. Gasca,
Generation of lattices of points for bivariate interpolation, Numer. Algorithms, {\bf 39} (2005) 69--79.

\bibitem{CGo06}
J. M. Carnicer and C. God\'es,
{Geometric characterization and generalized principal lattices}, J. Approx. Theory,  {\bf 143} (2006) 2--14.

\bibitem{CGo07}
J. M. Carnicer and C. God\'es,
{Geometric characterization of configurations with defect three,} in {\it Curve and Surface Fitting}, eds. Albert Cohen, Jean Louis Merrien, Larry L. Schumaker (Nashboro Press, Brentwood,  Avignon 2006, TN, 2007) 61--70.


\bibitem{CGo10}
J. M. Carnicer and C. God\'es,
{Configurations of nodes with defects greater than three.} J. Comput. Appl. Math., {\bf 233} (2010) 1640--1648.


\bibitem{CY77}
K. C. Chung and T. H. Yao,
{On lattices admitting unique Lagrange interpolations},
SIAM J. Numer. Anal., {\bf 14} (1977) 735--743.

\bibitem{GM82}
M. Gasca and J. I. Maeztu,
{On Lagrange and Hermite interpolation in $\mathbb{R}^k$},
Numer. Math., {\bf 39} (1982) 1--14.

\bibitem{HJZ14}
H. Hakopian, K. Jetter, and G. Zimmermann, The Gasca-Maeztu
conjecture for $n=5$,  Numer. Math.,  {\bf 127} (2014) 685--713.

\bibitem{HR}
H. Hakopian and L. Rafayelyan,
{On a generalization of Gasca-Maeztu conjecture},
New York J. Math., {\bf 21} (2015) 351--367.


\bibitem{HV}
H. Hakopian, V. Vardanyan,
On a correction of a property of $GC_n$ sets, Adv. Comput. Math.,  {\bf 45}, (2019) 311--325.

\bibitem{HV2}
H. Hakopian, V. Vardanyan,
On the usage of lines in $GC_n$ sets, Adv. Comput. Math.,  {\bf 45}, (2019) 2721--2743.


\end{thebibliography}
\end{document}